\documentclass[preprint,9pt,authoryear]{elsarticle}

\usepackage{amssymb}
\usepackage{amsthm}
\usepackage{graphics}
\usepackage{setspace}
\usepackage{amsmath}
\usepackage{mathrsfs}
\usepackage{textcomp}
\usepackage{wasysym}
\usepackage{helvet}
\usepackage{enumerate} 
\usepackage{pifont}
\usepackage{mathtools}
\usepackage{color}
\usepackage[table]{xcolor}
\renewcommand{\geq}{\geqslant}		%changes the => symbol
\newtheorem{theorem}{Theorem}[section]			%label theorems by section
\newtheorem{corollary}[theorem]{Corollary}		%Use the same numbering as for theorems
\newtheorem{lemma}[theorem]{Lemma}			%Use the same numbering as for theorems
\theoremstyle{definition}		%default definition styling

\newtheorem{example}{Example}[section]

\def\eref#1{$(\ref{#1})$}				

\def\egref#1{Example~$\ref{#1}$}		%Refrence an example
\def\lref#1{Lemma~$\ref{#1}$}		%Refrence a lemma
\def\tref#1{Theorem~$\ref{#1}$}		%Refrence a theorem
\def\cref#1{Corollary~$\ref{#1}$}		%Refrence a corollary
\usepackage{lineno}

\journal{Finite Fields and Applications}

\begin{document}

\begin{frontmatter}

\title{Permutation Polynomials of $\mathbb{F}_{q^2}$ : A Linear Algebraic Approach }

\author[label1]{Megha M. Kolhekar}
\author[label2]{Harish K. Pillai}
\address[label1]{meghakolhekar@ee.iitb.ac.in}
\address[label2]{hp@ee.iitb.ac.in}
\address{Department of Electrical Engineering, IIT Bombay, INDIA}

\begin{abstract}{
In this paper, we present a linear algebraic  approach to the study of permutation polynomials that arise from linear maps over a finite field $\mathbb{F}_{q^2}$. We study a particular class of permutation polynomials over $\mathbb{F}_{q^2}$, in the context of rank deficient and full rank linear maps over $\mathbb{F}_{q^2}$. We derive necessary and sufficient conditions under which the given class of polynomials are permutation polynomials. We further show that the number of such permutation polynomials can be easily enumerated. Only a subset of these permutation polynomials have been reported in literature earlier. It turns out that this class of permutation polynomials have compositional inverses of the same kind and we provide algorithms to evaluate the compositional inverses of most of these permutation polynomials. }
\end{abstract}


\end{frontmatter}
%\begin{linenumbers}

\section{Permutation Polynomials of Finite Fields}
Let $q$ be a prime power, $q=p^r$ and $\mathbb{F}_q$ be the finite field of cardinality $q$. A good introduction to the theory of finite fields is available in [\cite{Lidl:1986:IFF}]. Permutation polynomials are a topic of interest due to their applications in coding theory, cryptography, combinatorics and other engineering areas. A permutation polynomial (PP) of $\mathbb{F}_q$ is a bijection on $\mathbb{F}_q$ through the evaluation map. The relevant problems in the area concerning permutation polynomials of finite fields are listed in the seminal paper  [\cite{Lidl:1988:whendoes}] and there have been several attempts to solve these problems. Finding new classes of PPs is a challenging problem, especially classes which can be parameterized and are easy to construct. In this paper, we present a very large class of PPs that are easy to construct, simple to verify as being PPs, easy to enumerate and in addition this class of PPs is also closed under compositional inverses. We begin by discussing some background and literature about permutation polynomials that is relevant to the work presented in this paper.\\
%In this paper we consider `normalized' PPs, meaning monic permutation polynomials which map $0$ to $0$ through their evaluation. Note that if $f(x)$ is a normalized PP, then $\alpha f(x) +\beta$ is also a PP for $\alpha \in \mathbb{F}_q^*$ and $\beta \in \mathbb{F}_q$. Thus, if the count of normalized PPs is known, one can deduce that the actual count of PPs  is $q(q-1)$ times that.  \\

The connection between linearized polynomials and permutation polynomials has been explored way back in [\cite{evans1992}]. It turns out that all linearized polynomials that have full rank (when viewed as a linear map) are permutation polynomials.
One of the open problems stated in [\cite{Lidl:1988:whendoes}] is to find new classes of permutation polynomials.  There are several papers in literature that address this problem. One such paper is [\cite{2008arXiv0810.2830Z}], that discusses four constructions of permutation polynomials which are generalizations of constructions given in [\cite{MARCOS2011105}]. One of the most popular methods used for identifying new classes of PPs is through the application of AGW criterion. This criterion was developed by  building on material in [\cite{2008arXiv0810.2830Z}] and appeared in [\cite{AKBARY201151}]. The AGW criterion gives necessary and sufficient conditions for checking permutation polynomials, by reducing the problem to checking for permutation on a smaller set. The maps required to reduce the problem, needs to be built in a clever fashion. The authors construct several classes of permutation polynomials -- for example in Theorem $5.12$ in [\cite{AKBARY201151}], they construct two classes of permutation polynomials of $\mathbb{F}_{q^2}$ which are of the form (a) $f_{a,b,k}(x) = ax^q + bx + (x^q-x)^k$, where $a,b \in \mathbb{F}_q$ such that $a \not = \pm b$ and all positive even integers $k$ (b) $f_{a,k}(x) = ax^q + ax + (x^q-x)^k$, where $a \in \mathbb{F}_q^*$ and $p,k$ are odd and in addition $k$ is relatively prime to $q-1$. \\
In [\cite{Yuan2008PermutationPO}], the authors establish relationship between `Kloosterman polynomials' and permutation polynomials and provide some classes of the form $(x^p-x+\delta)^s +L(x)$, over fields of characteristic $2$ and $3$. The paper also provides some experimental results on the values of $s$ and $\delta$ for specific $m$ over $\mathbb{F}_{2^m}$ and $\mathbb{F}_{3^m}$. In [\cite{YUAN2011560}], the authors refer to the AGW criteria as `a powerful lemma' and provide generalizations using `$q$-polynomials' over $\mathbb{F}_q$ to construct permutation polynomials of $\mathbb{F}_{q^n}$. Their main result is about a class of permutation polynomials over $\mathbb{F}_{q^n}$ involving general polynomials of $\mathbb{F}_{q^n}[x]$ and linearized polynomials over $\mathbb{F}_q$. Necessary and sufficient conditions are provided for $f(x)$ of the given class to be a permutation polynomial of $\mathbb{F}_{q^n}$ in terms of some other polynomial being able to permute $\mathbb{F}_q$ and a system of equations having a common solution to be zero only. Theorem $6.4$ in [\cite{YUAN2011560}] is a generalized version of Theorem $5.12$ mentioned above from [\cite{AKBARY201151}], dealing with cases, that were not described earlier.\\
A lot of attention has centered around the class of permutation polynomials of the type $(x^q-x+\delta)^s+ax^q+bx$. In [\cite{ZHA2012781}], yet another approach to this class of permutation polynomials is seen. Specifically, the search experiments in [\cite{Yuan2008PermutationPO}] are provided with concrete proofs in [\cite{ZHA2012781}] for the class given by $f(x)=(x^{p^k}-x+\delta)^s+L(x)$. %The other class uses implicit/explicit piece-wise function characteristics over subsets of finite fields to establish credibility of the class as permutation polynomial. 
The results in [\cite{ZHA2012781}] were refined further in [\cite{YUAN201488}]. Authors Hellseth et.al. especially refine the theorem about the class $f(x)=(x^{p^k}-x+\delta)^{\frac{p^n+1}{2}}+x^{p^k}+x$ for any $\delta \in \mathbb{F}_{p^n}$. Some new classes of the form $(x^{p^k}-x+\delta)^s+x$ were obtained for specific values of the index $s$. \\
A significant reference thereafter is [\cite{Zheng:2016:LCP:2995737.2995773}], in which permutation polynomials over $\mathbb{F}_{q^2}$ with a `more general form' given by $f(x)=(ax^q+bx+c)^2+\phi((ax^q+bx+c)^{(q^2-1)/d})+ux^q+vx$ are discussed. The authors use the AGW criterion extensively to prove their results and unify some of the known permutation polynomial  classes. Several cases of $\phi(x)$ are also discussed. %Theorem $3$, Corollary $6$ and $7$ of [\cite{Zheng:2016:LCP:2995737.2995773}] are of specific interest to us in the present paper. %Necessary and sufficient conditions on the parameters of the class are derived for another special case of interest $\phi(x) \in \mathbb{F}_{q}[x]$. 
All the classes that are derived in [\cite{Zheng:2016:LCP:2995737.2995773}] are claimed to be `large'. \\
Recently, the permutation polynomials of the form $f(x)=(x^q-x+\delta)^s+ cx$ have taken center stage in papers like [\cite{ZHENG20191}]. This class is derived over $\mathbb{F}_{q^2}$ without any restriction on the parameter $\delta$. Permutation trinomials of the form $x^{qs}-x^s+cx$ over $\mathbb{F}_{q^2}$ are proposed. In [\cite{https://doi.org/10.48550/arxiv.2204.01545}], authors propose a generic construction of permutation polynomials, especially for permutation binomials and trinomials, which satisfy properties of some well known classes like $X^rh(X)^{q-1}$, over $\mathbb{F}_{q^2}$. The proposed algorithm claims to provide control over the coefficients of the permutation polynomial.\\ %In [\cite {https://doi.org/10.48550/arxiv.2206.04252}] the authors prove that every PP is an AGW PP. A general method to find compositional inverses through AGW criterion is proposed.\\
%In [\cite{Yuan2008PermutationPO}], permutation polynomials of the form $(x^p-x+\delta)^s +L(x)$ was  introduced. This work is over fields of characteristic $2$, $3$ and some more new classes are introduced in [\cite{Akbary08onsome}], [\cite{Zheng2015}], [\cite{zeng2015}], [\cite{Gupta2016NCP}], and [\cite{Yanping2017}]. A unified treatment using the popular AWG criterion to these types of permutation polynomials is attempted in [\cite{Yuan2011PermutationPO}]. In that paper, the authors also propose some specific classes of permutation polynomials over $\mathbb{F}_q$. In [\cite{2008arXiv0810.2830Z}]  Zieve has presented an additive analogue for PPs arising from cyclotomy. In recent years [\cite{reis2021permutation}] describes permutation polynomials from linearized decomposition.  This paper relates the problem of constructing PPs of $\mathbb{F}_{q^n}$ to the problem of factorizing $x^n-1$ in $\mathbb{F}_q[x]$.\\
In this paper, we study normalized PPs of $\mathbb{F}_{q^2}$ of the form $f(x) = g(s) + L(x)$, where $s$ is a rank $1$ linearized polynomial, $g \in \mathbb{F}_{q^2}[x]$ (the polynomial ring with coefficients from $\mathbb{F}_{q^2}$) and $L(x)$ is a linearized polynomial. We present an approach which is linear algebraic in nature thereby differing from the other approaches available in existing literature.  We claim that our approach is more general and is arguably inclusive of most permutation polynomial sub-classes presented in papers that deal with PPs of this particular shape. %We present a very generalized approach of looking at these PPs using linear algebraic methods and t
The novelty in the approach is in the ability to deduce much more information about these PPs as compared to the other methods in existing literature about these PPs. This includes ways of  counting the number of PPs,  apart from information about the compositional inverses of these PPs over $\mathbb{F}_{q^2}$. 
%%%%%%%%%%%%%%%%%%%%%%%%%%%%%%%%%%%%%%%%%%%%%%%%
\section{Preliminaries}
In this paper we consider `normalized' PPs, meaning monic permutation polynomials which map $0$ to $0$ through their evaluation. Note that if $f(x)$ is a normalized PP, then $\alpha f(x) +\beta$ is also a PP for $\alpha \in \mathbb{F}_q^*$ and $\beta \in \mathbb{F}_q$. Thus, by showing some property of a normalized PP, we are intrinsically providing information about a collection of $q(q-1)$ permutation polynomials. We shall often refer to these other PPs of the form $\alpha f(x) + \beta$ as PPs in the family associated to the normalized PP $f(x)$. One place where this introduction of normalized PPs turn out to be useful is for counting the number of PPs -- if the count of normalized PPs is known, one can deduce that the actual count of PPs  is $q(q-1)$ times the number of normalized PPs.  \\
$\mathbb{F}_{q^r}$ is a $r$-dimensional vector space over $\mathbb{F}_q$. In addition, there is a multiplicative group structure in $\mathbb{F}_{q^r}^*$ (the nonzero elements of $\mathbb{F}_{q^r}$). The natural addition of two elements in the field  $\mathbb{F}_{q^r}$ coincides with the vector addition of two vectors in the vector space. Utilizing the vector space structure one can look at the linear maps from $\mathbb{F}_{q^r}$ to itself. These linear maps  have a kernel and an image, which are subspaces of the $r$-dimensional vector space.  If this linear map is full rank, then the kernel is $\{0\}$. This implies that it is a one-to-one onto map and being a one-to-one, onto map on a finite set, it is a permutation. So, if a full rank linear map can be written as a polynomial, then one obtains a permutation polynomial. Linear maps on $\mathbb{F}_{q^r}$ are studied as linearized polynomials and is  well-elaborated in  [\cite{Lidl:1986:IFF}]. \\
A linearized polynomial of $\mathbb{F}_{q^r}$ is given as 
\begin{equation}\label{eq:linpoly}
L(x)= \sum_{i=0}^{r-1} \alpha_ix^{q^i}
\end{equation} where the coefficients are from the extension field $\mathbb{F}_{q^r}$ of $\mathbb{F}_q$.
Associated with every linearized polynomial, is a matrix known as the Dickson's matrix [\cite{WU201379}]. For a linearized polynomial of $\mathbb{F}_{q^r}$ as specified in \eref{eq:linpoly}, the associated Dickson's matrix is 
\begin{equation}\label{eq:dickson_gen}
\mathcal{D}(L) = \begin{pmatrix}
\alpha_0 & \alpha_1 & \ldots & \alpha_{r-1}\\
\alpha_{r-1}^q & \alpha_0^q & \ldots & \alpha_{r-2}^q\\
\vdots & \vdots & \vdots & \vdots \\
\alpha_1^{q^{r-1}} & \alpha_2^{q^{r-1}} & \ldots & \alpha_0^{q^{r-1}}
\end{pmatrix}
\end{equation}
The rank of the linear map induced by the  linearized polynomial is precisely the rank of the associated Dickson's matrix. %So, a linearized polynomial that corresponds to a full rank linear map, is a permutation polynomial of $\mathbb{F}_{p^r}$. \\
Linearized polynomials and permutations of finite fields from these linearized polynomials have been a topic of interest in many papers. Permutations arising from the linearized polynomials can be studied in a more systematic way, using linear algebraic techniques, which is the main thrust of this paper. We shall consider the second extension $\mathbb{F}_{q^2}$ of $\mathbb{F}_q$ and in this case, the linearized polynomials are of the form 
\[
L(x)=\alpha_1 x^q+\alpha_0 x
\]
where $\alpha_0, \alpha_1 \in \mathbb{F}_{q^2}$. Except the zero polynomial, these linearized polynomials are either rank $1$ or rank $2$.  The associated Dickson's matrix for this linearized polynomial is given by

\begin{equation}\label{eq:dickson}
\mathcal{D}(L) = \begin{pmatrix}
\alpha_0 & \alpha_1\\
\alpha_1^q & \alpha_0^{q}
\end{pmatrix}
\end{equation}
It is clear that the linearized polynomials for which $\alpha_0^{q+1}\neq \alpha_1^{q+1}$ are  permutation polynomials of $\mathbb{F}_{q^2}$. On the other hand, there is a non trivial kernel associated to rank $1$ linear maps and therefore they do not correspond to permutation polynomials. We denote the kernel and image of a linear map $L(x)$ by $\ker L(x)$ and ${\rm im}\,\,L(x)$ respectively. \\
It is a well known fact that associated with  every permutation is a inverse permutation. Inverse permutation polynomials are  `compositional inverses'.  Two permutation polynomials $f(x)$ and $h(x)$ are compositional inverses of each other if $f(x) \circ h(x) =h(x) \circ f(x) =x ~mod~(x^{q^2}-x)$. The `$\circ$' operation is defined as $f(x) \circ h(x) = f(h(x))$. The compositional inverse of a rank $2$ linearized polynomial is given by
\begin{lemma}\label{lem:lin_comp}
Let $L(x)=\alpha x^q +\beta x $ be a rank $2$ linearized polynomial (and therefore a permutation polynomial) of $\mathbb{F}_{q^2}$. Then $h(x)= \gamma x^q+ \epsilon x$ is the compositional inverse of $L(x)$ with
\[
\gamma=\frac{-\alpha}{\beta^{q+1}-\alpha^{q+1}}
\]
\[
\epsilon=\frac{\beta^q}{\beta^{q+1}-\alpha^{q+1}}
\]
\end{lemma}
\begin{proof}

\[
f \circ h =\alpha(\gamma x^q+\epsilon x)^q +\beta (\gamma x^q +\epsilon x)
\]
We use the well known Binomial Theorem here [Theorem $1.46$, \cite {Lidl:1986:IFF}]. 
\begin{equation}\label{eq:f_circ_h}
f \circ h= \alpha (\gamma ^q x+\epsilon ^q x^q) +\beta (\gamma x^q + \epsilon x)
\end{equation}
Note that $(\beta^{q+1}-\alpha^{q+1})^q=\beta^{q+1}-\alpha^{q+1}$ and therefore
\[
\gamma^q=\frac{-\alpha^q}{\beta^{q+1}-\alpha^{q+1}}
\]
\[
\epsilon^q=\frac{\beta}{\beta^{q+1}-\alpha^{q+1}}
\]
The result is clear by plugging in $\gamma, \epsilon, \gamma^q$ and $\epsilon^q$ in \eref{eq:f_circ_h}.
\end{proof}
The counts for number of rank $1$ and rank $2$ monic linearized polynomials in $\mathbb{F}_{q^2}[x]$ are obtained easily. %, denoted by $L_1$ and $L_2$ respectively. 
\begin{lemma} \label{lem:rank1_L}
The number of monic rank $1$ linearized polynomials over $\mathbb{F}_{q^2}$ is $q+1$.
\end{lemma}
\begin{proof}
From \eref{eq:dickson}, assuming $\alpha_1 = 1$ for monic polynomials, one concludes that rank $1$ linearized polynomials are the ones with  $\alpha_0^{q+1}=1$. In other words,  $\alpha_0$ is a root of $x^{q+1}-1$. This polynomial has distinct roots in $\mathbb{F}_{q^2}$ [Theorem $2.42$, \cite {Lidl:1986:IFF}]. These roots are of the form $\omega^{i(q-1)}$ where $\omega$ is a primitive element and $i=1,2,\ldots,q+1$. Therefore, there are $q+1$ distinct rank $1$ monic linearized polynomials in $\mathbb{F}_{q^2}$.
\end{proof}
The consequence of this is the following.
\begin{lemma}\label{lem:rank2_L}
The number of monic linearized polynomials of rank $2$ over $\mathbb{F}_{q^2}$ %and therefore the number of permutation polynomials of this type, 
is $q^2-q$.
\end{lemma}
\begin{proof}
The number of possible coefficients $\alpha_0$ in monic linearized polynomials $x^q+\alpha_0 x$ is $q^2$. Therefore, total number of monic linearized polynomials including the polynomial $x$ is $q^2+1$. Out of them, $q+1$ are rank $1$ from \lref{lem:rank1_L}.  Hence, the number of rank $2$ linearized polynomials is $q^2-q$.
%Since, a rank $2$ linearized polynomials of $\mathbb{F}_{q^2}$ is a permutation polynomial, this is also the number of monic linearized PPs over $\mathbb{F}_{q^2}$.
\end{proof}

%%%%%%%%%%%%%%%%%%%%%%%%%%%%%%%%%%%%%%%%%%%%%%%%%%%%%%%%%%%%%%%%%%%%
\section{Permutation Polynomials generated from the Linearized Polynomials of $\mathbb{F}_{q^2}$} 
In this section, permutation polynomials of the type  $f=(x^q+\delta_ix)^m+ L(x)$ are studied where $\delta_i$ is a $(q+1)$-th root of unity and $L(x)$ is a linearized polynomial of $\mathbb{F}_{q^2}$. The following theorem provides a sufficiency condition for existence of PPs, using which a lower bound on the number of normalized PPs of the shape $f=(x^q+\delta_ix)^m + L(x)$ is calculated. For the sake of brevity, we shall use $s = x^q + \delta_ix$.  \\ %states the existence of a minimum number of PPs of this form. \\
%%%%%%%%%%%%%%%%%%%%%%%%%%%%%%%%%%%%%
\begin{theorem}\label{thm:lowerbound}
The polynomial $f = s^m + L(x)$ is a permutation polynomial of $\mathbb{F}_{q^2}$ if $L(x)$ is a rank $2$ linearized polynomial of $\mathbb{F}_{q^2}$ that maps the kernel of $s$ to the subspace containing the image of $s^m$.
\end{theorem}
\begin{proof}
As $s = x^q + \delta_ix$ where $\delta_i$ is a $(q+1)$-th root of unity, therefore $s$ is a rank $1$ linearized polynomial of $\mathbb{F}_{q^2}$. Let $u \in \mathbb{F}_{q^2}$ be an element in the kernel of $s$. Consider a basis $\{u,v\}$ of $\mathbb{F}_{q^2}$. So any $w_i \in \mathbb{F}_{q^2}$ can be written as $w_i = \alpha_i u + \beta_i v$ where $\alpha_i,\beta_i \in \mathbb{F}_q$. For $f = s^m + L(x)$, observe that $f(w_i)  = s^m(\beta_i v) + L(w_i) = \beta_i^ms^m(v) + L(w_i) = \beta_i^mz + L(w_i)$, where $z = s^m(v)$. For $f$ to be a permutation polynomial, it is enough to demonstrate that $f(w_i) \not= f(w_j)$ whenever $w_i \not= w_j$ for $w_i, w_j \in \mathbb{F}_{q^2}$.

Consider $w_i,w_j \in \mathbb{F}_{q^2}$ such that $\beta_i = \beta_j$ and $\alpha_i \not= \alpha_j$. Then $f(w_i) = \beta_i^m z + L(w_i)$ and $f(w_j) = \beta_i^m z + L(w_j)$. If $L(x)$ is a rank $2$ linearized polynomial, then $L(w_i) \not= L(w_j)$ and this ensures that $f(w_i) \not= f(w_j)$.

Now consider $w_i, w_j \in \mathbb{F}_{q^2}$ such that $\beta_i \not= \beta_j$. If $f(w_i) = f(w_j)$, then rearranging the terms one gets \[ (\beta_i^m - \beta_j^m) z = L(w_j) - L(w_i) = (\alpha_j - \alpha_i)L(u) + (\beta_j - \beta_i)L(v) \] 
But if $L(u) = \gamma z$ where $\gamma \in \mathbb{F}_q$, then it is clear that the above equation cannot hold and therefore $f(w_i) \not= f(w_j)$. The condition that $L(u) = \gamma z$ is precisely ensuring that the kernel of $s$ (the line containing $u$) maps to the line containing the image of $s^m$ (all elements of the type $\beta_i^m z$, which is contained in the line defined by $z$).
\end{proof}
%%%%%%%%%%%%%%%%%%%%%%%%%%%%%%%%%%%%%%%%%%%%%%%%%%%%%%%
A few comments are in order about \tref{thm:lowerbound}. Firstly $s$ stands for any one of the $(q+1)$ monic rank $1$ linearized polynomials. Secondly, though the theorem holds for all possible values of $m$, it is meaningless to consider all of these $m$. For the case $m=1$, \tref{thm:lowerbound} just tells us a condition under which a rank $2$ linearized polynomial can be added to a monic rank $1$ linearized polynomial to obtain another rank $2$ linearized polynomial. For the case $m=q$, we have $s^q = (x^q + \delta_ix)^q = \delta_i^q(x^q + \delta_ix) = \delta_i^qs$, which is a scalar multiple of $s$. As a result, for all $m \ge q$, \tref{thm:lowerbound} generates PPs which are scalar multiples of PPs generated when $1 \le m < q$. Therefore technically, new normalized PPs are generated only when $2 \le m \le (q-1)$.\\ Note that \tref{thm:lowerbound} is a generalization of part (a) Theorem $5.12$ of [\cite{AKBARY201151}] where only even values of $m$ are considered and the special case of $\delta_i = -1$ is the only case under consideration. Theorem $6.4$ of  [\cite{YUAN2011560}], is also restricted to the case $\delta_i = -1$, whereas \tref{thm:lowerbound} is far more general.
\tref{thm:lowerbound} can be utilized to establish a lower bound on the number of normalized PPs of the kind $f=(x^q+\delta_i x)^m+L(x)$.
%%%%%%%%%%%%%%%%%%%%%%%%%%%%%%%%%%%%%%%%%%%%%%%%%%%%%%5
\begin{corollary}\label{cor:lowerbound}
Consider $(x^q + \delta_ix)^m$ for $\delta_i$ which is a $(q+1)$-th root of unity and any $m$ such that $2 \le m < q$. Then there are at least $q(q-1)^2$ normalized PPs having the shape $(x^q + \delta_ix)^m + L(x)$ where $L(x)$ is a rank 2 linearized polynomial.  
\end{corollary}
\begin{proof}
\tref{thm:lowerbound} prescribes a sufficient condition for $L(x)$ that make $(x^q + \delta_ix)^m + L(x)$ a normalized PP. Using notation from the proof of \tref{thm:lowerbound}, the linear maps corresponding to $L(x)$ should map $u$ to $\gamma z$. Any linear map is completely specified by specifying where a basis maps to. For the basis $\{u,v\}$ (as in the proof of \tref{thm:lowerbound}), there are $q-1$ choices for $\gamma$ which specifies $L(u) = \gamma z$ and there are $q^2-q$ vectors that are possible  choices for $L(v)$. This makes a total of $q(q-1)^2$ possibilities for $L(x)$.  
\end{proof}
%%%%%%%%%%%%%%%%%%%%%%%%%%%%%%%%%%%%%%%%%%%%%%%%%%%%%%%%%%%%%%%
At this point, we make an estimate of the total number of PPs that one can obtain which has the form $(x^q + \delta_ix)^m + L(x)$. From \cref{cor:lowerbound} we see that the number of normalized PPs for each $\delta_i$ and each permissible $m$ is $q(q-1)^2$. There are $q+1$ choices for $\delta_i$ (as these are the $(q+1)$-th roots of unity). Further there are $q-2$ possible choices for $m$ (as $2 \le m < q$). This makes a total of $q(q-1)^2(q+1)(q-2)$ normalized PPs. As each normalized PP $f$ can be associated with $q^2(q^2-1)$ PPs of the form $\alpha f(x) + \beta$ where $\alpha \in \mathbb{F}^*_{q^2}, \beta \in \mathbb{F}_{q^2}$, the estimate for PPs obtained by the above construction reaches  $q^3(q-1)^3(q+1)^2(q-2)$ -- a ninth degree polynomial in $q$. 
\begin{example}
Let $q$ be the prime $p=7$. Over $\mathbb{F}_{7^2}$, $s=(x^7+\delta_ix)$.  The total number of rank $2$ linearized polynomials over $\mathbb{F}_{7^2}$ is $2016$. Out  of them $p(p-1)^2=252$ rank $2$ linearized polynomials satisfy the required condition for each $s^m$.  There are $8$ choices for $\delta_i$ and $5$ choices for $m$. %These rank $2$ linearized polynomials can combine with $s^2, s^3, s^4, s^5$ and $s^6$ to form the permutation polynomials of this shape.   And, the $s$ itself can take $8$ forms as there are $8$ elements $\delta_i$. 
Therefore a lower bound for the total number of normalized PPs of the kind $(x^7 + \delta_ix)^m + L(x)$ is $10080$. If we count all PPs from the associated families of each of these normalized PPs, then the number of PPs become  $10080\times 49\times 48=23708160$. 
\end{example}
Some more examples of counts of normalized PPs for small fields are collected together in Table~1.
\begin{table}[!ht]
\begin{center}
\caption{Number of Normalized PPs $f(x)=s^m+L(x)$ for one $s=(x^p+\delta_i x)$ and rank $2$ $L(x)$ for small fields}
\vspace{0.2cm}
\label{tab:Table1}
\begin{tabular}{|c|c|c|c|c|c|c|c|}
\hline
Field & $\mathbb{F}_{4^2}$ & $\mathbb{F}_{5^2}$ &$\mathbb{F}_{7^2}$ & $\mathbb{F}_{8^2}$  & $\mathbb{F}_{9^2}$  & $\mathbb{F}_{11^2}$ & $\mathbb{F}_{13^2}$ \\
\hline
$index~m=2$ & $36$ & $80$ & $252$  & $392$ & $576$ & $1100$ & $1872 $ \\
\hline
$index~m=3$ & $36$ & $80$ & $252$ & $392$ & $3168$ & $1100$ & $1872$\\
\hline
$index~m=4$ & $--$ & $80$ & $840$ & $392$ & $576$ & $1100$  & $1872$\\
\hline
$index~m=5$ & $--$ & $--$ & $252$ & $392$ & $1872$ & $1100$  & $1872$\\
\hline
$index~m=6$ & $--$ & $--$ & $252$ & $392$ & $576$ & $5940$  & $1872$\\
\hline
$index~m=7$ & $--$ & $--$ & $--$ &  $392$ & $576$ & $1100$  & $9984$ \\
\hline
$index~m=8$ & $--$ & $--$ & $--$ & $--$ & $576$ & $1100$ & $1872$\\
\hline
$index~m=9$ & $--$ & $--$ & $--$  & $--$ & $--$ & $1100$ & $1872$\\
\hline
$index~m=10$ & $--$ & $--$  & $--$ & $--$ & $--$ & $1100$ & $1872$  \\
\hline
$index~m=11$ & $--$ & $--$  & $--$ & $--$ & $--$ & $--$ & $1872$\\
\hline
$index~m=12$ & $--$ & $--$  & $--$ & $--$ & $--$ & $--$ & $1872$\\
\hline
\end{tabular}
\end{center}
\end{table}
%%%%%%%%%%%%%%%%%%%%%%%%%%%%%%%%%%%%%%%%%%%%%%%%%%%
%%%%%%%%%%%%%%%%%%%%%%%%%%%%%%%%

As one can see, there is a general agreement in the numbers predicted by \cref{cor:lowerbound} with those in Table~$1$. The special cases where the numbers do not agree shall be dealt with a bit later in the paper.

\tref{thm:lowerbound} rests crucially on the fact that the linearized polynomial $L(x)$ that is added to $(x^q + \delta_ix)^m$ corresponds to a rank $2$ linear map. We now explore if one can obtain permutation polynomials by taking $L(x)$ that correspond to rank $1$ linear maps. Once again let $s = (x^q + \delta_ix)$, where $\delta_i$ is a $(q+1)$-th root of unity.   
\begin{theorem}\label{thm:coprime}
$f= s^m + L(x)$ (with $L(x)$ a rank $1$ linear map) is a permutation polynomial of $\mathbb{F}_{q^2}$ only if $gcd(m,q-1)=1$. Further, $f$ is a permutation polynomial precisely when  $\ker s \not= \ker L(x)$ and  ${\rm im}\, s^m \not= {\rm im}\, L(x)$. 
\end{theorem}
\begin{proof}
Note that if $L(x)$ is a rank $1$ linear map, then the image of this linear map is one dimensional. Thus the cardinality of the image of $L(x)$ is $q$. Observe that the cardinality of the image of $s$ is also $q$ (as $s$ is also a rank $1$ linear map). Therefore the cardinality of  image of $s^m$ is at most $q$. The image of $f$ consists of all possible sums of elements in the image of $L(x)$ with elements in the image of $s^m$. The cardinality of the image of $f$ needs to be $q^2$ for $f$ to be a permutation polynomial. This is only possible if cardinality of image of $s^m$ is precisely equal to $q$. 

If $v \in \mathbb{F}_{q^2}$ is in the image of $s^m$, then the other elements in the image of $s^m$ are of the form $\alpha^m v$ where $\alpha \in \mathbb{F}_q$. Thus the image of $s^m$ has cardinality $q$ iff $x^m$ is a permutation polynomial of $\mathbb{F}_q$. It is well known that $x^m$ is a permutation polynomial of $\mathbb{F}_q$ precisely when gcd$(m,q-1) = 1$.

If $\ker s = \ker L(x)$, then all the elements $v \in \ker L(x)$ map to zero under $f$, that is, $f(v) = 0$. Therefore, $f$ cannot be a permutation polynomial, as multiple elements map to zero under $f$. Similarly, if ${\rm im}\, s^m = {\rm im}\, L(x)$, then every element in the image of $f$ lies within the image of $L(x)$. Since image  $L(x)$ is one dimensional, $f$ cannot be a permutation polynomial. Conversely, when ${\rm im}\, s^m \not= {\rm im}\, L(x)$, then $|{\rm im}\, s^m| = q$ and $|{\rm im}\,L(x)| = q$. Further ${\rm im}\, s^m \cap {\rm im}\, L(x) = \{ 0 \}$ and this ensures that ${\rm im}\, f = \mathbb{F}_{q^2}$.
\end{proof}
%%%%%%%%%%%%%%%%%%%%%%%%%%%%%%%%%%%%%%%%%%%%%%%%%%%%%%%%%%%%%%%%%%%%%%%%%
\tref{thm:coprime} provides a necessary and sufficient condition for linearized polynomials of rank $1$ to be added to $s^m$ (for $m>1$) to form a normalized PP. For $m=1$, \tref{thm:coprime} holds but the normalized PPs so obtained are precisely the linearized polynomials of rank 2. Note that \tref{thm:coprime} is a generalization of part (b) of Theorem $5.12$ of [\cite{AKBARY201151}], where only the special case of $\delta_i = -1$ is considered and $L(x)$ was taken to be rank $1$ matrices of the special form $\alpha (x^q + x)$, with $\alpha \in \mathbb{F}_q^*$. As one can see from \tref{thm:coprime}, there is a much larger number of normalized PPs of this form than considered in Theorem $5.12$ of [\cite{AKBARY201151}]. In the same way, \tref{thm:coprime} is also a generalization of part (b) of Theorem~6.4 of [\cite{YUAN2011560}], since there again $s$ is restricted to the special case of $x^q - x$, whereas $L(x)$ of the form $\alpha x^q + \alpha^q x$ essentially translates to permissible rank $1$ linearized polynomials that do not map to the image of $(x^q-x)^m$. The result in Theorem~6.4 of [\cite{YUAN2011560} only accounts for $q^2-q$ of the total number of normalized PPs of this form associated with $s = x^q -x$. The precise count of normalized PPs  arising out of \tref{thm:coprime} is 
\begin{corollary}\label{cor:coprime_count}
Consider $s =(x^q + \delta_ix)$ for $\delta_i$ which is a $(q+1)$-th root of unity and a $m > 1$ such that $gcd(m,q-1)=1$. Then there are $q^2(q-1)$ normalized PPs having the shape $s^m + L(x)$ where $L(x)$ is a rank $1$ linearized polynomial.  
\end{corollary}
\begin{proof}
As $\ker s \not= \ker L(x)$, therefore rank $1$ linearized polynomial should be of the form $L(x) = \alpha (x^q + \delta_jx)$ with $\alpha \in \mathbb{F}_{q^2}$ and $\delta_j \not= \delta_i$, where $\delta_j$ is also a $(q+1)$-th root of unity. The number of choices for $\delta_j$ is $q$. As ${\rm im}\, s^m \not= {\rm im}\, L(x)$, therefore the number of choices of $\alpha$ is $q^2-q$. Thus the total number of normalized PPs of the kind $s^m + L(x)$ where $L(x)$ is rank $1$ and $m>1$ coprime to $q-1$, is $q^2(q-1)$. 
\end{proof}
%%%%%%%%%%%%%%%%%%%%%%%%%%%%%%%%%%%%%%%%%%%%%%%%%%%%%%%%%%%%%%%%%%%%%%%%%%%%%%%%%%%%%
The number of $m$ that are coprime to $q-1$ is given by Euler's Totient function $\phi$. As $\phi$ also counts $1$ as coprime to any $q-1$, the number of indices $m$ that correspond to \cref{cor:coprime_count} is $\phi(q-1)-1$. Thus the total number of normalized PPs arising out of \cref{cor:coprime_count} is $(\phi(q-1)-1)(q+1)q^2(q-1)$ and the total number of associated PPs are $(\phi(q-1)-1)q^4(q-1)^2(q+1)^2$ -- a shade smaller than the number of PPs estimated from \cref{cor:lowerbound}.
\begin{example}
Over $\mathbb{F}_{11^2}$, there are three possible indices for $m$, namely $m=3,7,9$, such that $gcd(m,q-1)=1$ and $m \not = 1$. The number of rank $1$ $L(x)$ that make $s^m+L(x)$ a permutation polynomial (for a specific $s$ and $m$) is $q^2(q-1)=1210$. The total number of normalized PPs obtained via this construction is dictated by the number of distinct $\delta_i$s that make the $s$ (which in this case is $12$) and the number of distinct $m$ (in this case, $\phi(q-1)-1 = \phi(10)-1 = 3$), which amounts to  $12\times3\times1210=43560$. If we count all associated PPs of each of these normalized PPs, then the number of PPs become $43560\times121\times120=632491200$. 
\end{example}
Table~$2$ gives the number of normalized PPs (for some small fields) of the shape $f(x) = s^m + L(x)$ that one can obtain for each $s = x^q + \delta_ix$ and each permissible index $m$ that permits $L(x)$ to be a rank 1 linearized polynomial. \\
\begin{table}[h!]
\begin{center}
\caption{Number of normalized PPs $f(x)=s^m+L(x)$ for one $s=(x^p+\delta_i x)$, one index $m$ and rank $1$ $L(x)$ for small fields}
\vspace{0.2cm}
\label{tab:Table2}
\begin{tabular}{|c|c|c|c|c|c|c|c|}
\hline
Field & $\mathbb{F}_{4^2}$ & $\mathbb{F}_{5^2}$ & $\mathbb{F}_{7^2}$ & $\mathbb{F}_{8^2}$  & $\mathbb{F}_{9^2}$  & $\mathbb{F}_{11^2}$ & $\mathbb{F}_{13^2}$ \\
\hline
indices:  & $2$ & $3$ & $5$ &$2,3,4,5,6$  & $3,5,7$  &  $3,7,9$  & $ 5,7,11$\\
\hline
count: & $48$ & $100$ & $294$ & $448$ & $648$ & $1210$ & $2028$ \\
\hline
\end{tabular}
\end{center}
\end{table}
%%%%%%%%%%%%%%%%%%%%%%%%%%%%%%%%%%%%%%%%%%%%%%%%%%
\\
While \tref{thm:coprime} gives necessary and sufficient conditions for $s^m + L(x)$ to be a normalized PP with $L(x)$ being a rank $1$ linearized polynomial, \tref{thm:lowerbound} only gives us a sufficiency condition for $s^m + L(x)$ to be a normalized PP with $L(x)$ being a rank $2$ linearized polynomial. We now close this gap with the following theorem. Once again we let $s = x^q + \delta_ix$, where $\delta_i$ is a $(q+1)$-th root of unity.
%%%%%%%%%%%%%%%%%%%%%%%%%%%%%%%%%
\begin{theorem}\label{thm:Fp_to_Fpsquare}
If $x^m + \gamma x$ (for $m>1$) is a permutation polynomial of $\mathbb{F}_q$ for some $\gamma \in \mathbb{F}_q$, then there exist normalized PPs of the form $s^m + L(x)$  where $L(x)$ is a rank $2$ linearized polynomial that does not map the kernel of $s$ to the subspace containing the image of $s^m$.

If there are $k$ distinct $\gamma \in \mathbb{F}_q$ such that $x^m + \gamma x$ is a PP of $\mathbb{F}_q$, then there are $kq^2(q-1)$ distinct rank $2$ linearized polynomials $L(x)$ that do not map the kernel of $s$ to the subspace containing the image of $s^m$, such that $s^m + L(x)$ are normalized PPs of $\mathbb{F}_{q^2}$.
 \end{theorem}
\begin{proof}
Like in the proof of \tref{thm:lowerbound}, let us consider $\{u,v\}$ to be a basis of $\mathbb{F}_{q^2}$ such that $u \in \ker s$. Let $s^m(v) = z \in \mathbb{F}_{q^2}$ and therefore all elements in ${\rm im}\, s^m$ are given by $\beta^m z$ for $\beta \in \mathbb{F}_q$.

\tref{thm:lowerbound} takes care of rank $2$ linearized polynomials that map $\ker s$ to the subspace containing ${\rm im}\, s^m$. Let us therefore consider a rank $2$ linearized polynomial $L(x)$ such that $L(u)$ does not belong to the span of ${\rm im}\, s^m$. Let $\alpha u + v = b \in \mathbb{F}_{q^2}$ (with $\alpha \in \mathbb{F}_q$) and $L(x)$ be such that $L(b) = z$. Note that $s^m(b) = s^m(\alpha u + v) = s^m(v) = z$. Then we claim that $f = s^m + \gamma L(x)$ is a permutation polynomial of $\mathbb{F}_{q^2}$ if $x^m + \gamma x$ is a permutation polynomial of $\mathbb{F}_q$.

It is enough to show that $f(w_1) \not= f(w_2)$ whenever $w_1 \not= w_2$, with $w_1, w_2 \in \mathbb{F}_{q^2}$. Let $w_i = \alpha_i u + \beta_i v$ with $\alpha_i, \beta_i \in \mathbb{F}_q$ for $i = 1,2$. We can re-write this as $w_i = (\alpha_i - \beta_i \alpha)u + \beta_i b$. Now, \begin{eqnarray*} f(w_i) & = & s^m(w_i) + \gamma L(w_i) \\ &  = & \beta_i^m z + \gamma (\alpha_i - \beta_i \alpha)L(u) + \gamma \beta_i L(b) \\ & = & (\beta_i^m + \gamma \beta_i)z + \gamma(\alpha_i - \beta_i\alpha)L(u) \end{eqnarray*} 
As $L(x)$ is a rank $2$ linear map, the vectors $z, L(u)$ are linearly independent and so form a basis for $\mathbb{F}_{q^2}$. Now if $f(w_1) = f(w_2)$, then $\beta_1^m + \gamma \beta_1  = \beta_2^m + \gamma \beta_2$ and $\gamma(\alpha_1 - \beta_1\alpha) = \gamma(\alpha_2 - \beta_2\alpha)$. As $x^m + \gamma x$ is a permutation polynomial of $\mathbb{F}_q$, the first equality forces $\beta_1 = \beta_2$ and using this in the second equality gives $\alpha_1 = \alpha_2$. This concludes the proof of the claim that $f = s^m + \gamma L(x)$ is a permutation polynomial of $\mathbb{F}_{q^2}$.

For counting the number of normalized PPs of this kind, note that there are $q$ vectors of the form $b = \alpha u + v$. We now consider the choice for the linear map $L(x)$. If $L(b) = z$, then it is enough to choose a vector for $L(u)$ from the $q^2 - q$ vectors linearly independent of $z$ to get a rank $2$ linearized polynomial $L(x)$. Thus, there are a possible $q^2(q-1)$ choices available for $L(x)$ that can form a permutation polynomial of the shape  $s^m + \gamma L(x)$ for every $\gamma \in \mathbb{F}_q$ such that $x^m + \gamma x$ is a PP of $\mathbb{F}_q$. Thus, the total number of distinct rank $2$ linearized polynomials $L(x)$  that do not map the kernel of $s$ to the span of image of $s^m$ and form a normalized PP of the shape $s^m + L(x)$ is $kq^2(q-1)$ if there are $k$ distinct $\gamma \in \mathbb{F}_q$ such that $x^m + \gamma x$ is a PP of $\mathbb{F}_q$.
\end{proof}
The upper bound on the number of permutation polynomials of the shape $s^m + L(x)$ over $\mathbb{F}_{q^2}$ is dependent on the number of permutation polynomials of $\mathbb{F}_q$ that are binomials of the form $x^m + \gamma x$. We list a few examples of $\mathbb{F}_{q^2}$ and $m$ where such PPs arise from permutation binomials of the base field. %binomials not known, there is definitely a greater number of them from what we have counted earlier. It is experimentally seen for small fields, that, there always exists a permutation polynomial of the form $x^{(q+1)/2}+\gamma x$ over $\mathbb{F}_q$, for odd prime characteristic. Some more examples follow here.
\begin{example}\label{ex:add_count}

\begin{enumerate}
\item Over $\mathbb{F}_{7}$ there are two degree $4$ permutation polynomials of the form $x^4+\gamma x$ and they are $x^4+3x$ and $x^4+4x$. Corresponding to these two base-field permutations, there are $2\times7^2\times6 = 588$ additional normalized PPs having the shape $s^4 + L(x)$ for each $s = x^7 +\delta_ix$ in $\mathbb{F}_{7^2}$. These normalized PPs $s^4 + L(x)$ are in addition to the ones stated in \tref{thm:lowerbound} which take $ker(s)$ to span of $im(s^m)$. This explains the count of $840$ in Table $1$ for the case of $m=4$. 
\item Over $\mathbb{F}_9$, there are $4$ degree $3$ and $2$ degree $5$ permutation binomials. This explains the additional count of $4\times9^2\times8=2592$ for $m=3$ case in Table $1$. For $m=5$, additional number is $2\times9^2\times8=1296$.
\item While in $\mathbb{F}_{19}$, there are $3$ degree $7$, $8$ degree $10$ and $3$ degree $13$ permutation polynomials of the form $x^m+\gamma x$. The number of additional  normalized permutation polynomials in $\mathbb{F}_{19^2}$ contributed by these permutation polynomials of the base field are 
\begin{eqnarray}
s^7 +L(x):~~~ 3\times19^2\times18=19494 \nonumber\\
s^{10}+L(x):~~~8\times19^2\times18=51984 \nonumber \\
s^{13}+L(x): ~~~3\times19^2\times18=19494 \nonumber
\end{eqnarray}
\end{enumerate}
\end{example}
It would be interesting to find out for which indices $m$ can a binomial of the form $x^m + \gamma x$ be a permutation polynomial of $\mathbb{F}_q$. A great deal of analysis on permutation binomials of the type $f=x^m+\gamma x^n$ with some conditions on $m$ and $n$ with respect to the cardinality of the field is given in [\cite{masuda2009permutation}]. The paper describes many cases of this general class of `permutation binomials' and specifies conditions that the indices $m$ and $n$ must satisfy if $f$ is a permutation binomial of $\mathbb{F}_q$. %\textcolor{red}{Therefore, we have only verified experimentally for some small  fields with cardinality $<100$.} 

We have observed that for fields of odd characteristic, there were permutation polynomials of the form $x^m + \gamma x$ for $m= (q+1)/2$ (this was not observed in $\mathbb{F}_3$ and $\mathbb{F}_5$).  %that there are several degrees other than $(p+1)/2$  available in this class of PPs over $\mathbb{F}_p$; of the form $f=x^m+kx$. We have observed consistency in the counts of the PPs over $\mathbb{F}_{p^2}$ that we have presented, for these degrees. A very specific case is $m=(q+1)/2$ (odd prime characteristic)  for which we have established  the count $N$ is presented in the following.
%%%%%%%%%%%%%%%%%%%%%%%%%%%%%%%%%%%%%%%%%%%%%%%%%%%
%\subsection{Count for the case $m=(q+1)/2$ for Fields with Odd Characteristic}
This observation coincided with Theorem $1$ in  \cite{carlitz62} which gave conditions that the element $\gamma$ should satisfy in order that $f(x)=x^{(q+1)/2}+\gamma x$ is a permutation binomial of $\mathbb{F}_q$. We paraphrase that theorem below.

\begin{theorem}\label{thm:carlitz_Fp} \cite{carlitz62}: The polynomial
\[
f(x)=x^{m}+\gamma x ~~~~~m = (q+1)/2
\]
with $\gamma$ defined by
\[
\gamma=\frac{c^2+1}{c^2-1}
\]
for $c^2\neq \pm 1$ or $0$; is a permutation polynomial of $GF(q)$ provided $q\geq 7$. However, it is not a permutation polynomial of $GF(q^r)$ for $r>1$.
\end{theorem}
One can count the number of permutation binomials of the form $f(x)=x^{(q+1)/2}+\gamma x$ using conditions given in \tref{thm:carlitz_Fp}. Every field $\mathbb{F}_q$ of odd characteristic has $(q-1)/2$ nonzero squares. Whether $-1$ is a square in the field $\mathbb{F}_q$ depends upon the nature of $q$. This gives a handle to count the number of normalized PPs of the form $s^m + L(x)$ for $m = (q+1)/2$.
%%%%%%%%%%%%%%%%%%%%%%%%%%%%%%%%%%%%%%%%%%%%%%%%%%%%%%%%%%%%%%%%%%%%%%5
%%%%%%%%%%%%%%%%%%%%%%%%%%%%%%%%%%%%%%%%%%%%%%%%%%%%%%%%%%%%%%%%%%%%%%%%%%%%%%%%%%%%%%
\begin{lemma}\label{lem:additinal_count}
The number of additional permutation polynomials of the type $f(x)=(x^q+\delta_ix)^{(q+1)/2}+ L(x)$, as stated in \tref{thm:Fp_to_Fpsquare} is $Nq^2(q-1)$, where 
\begin{itemize}
\item $N=\frac{q-3}{2}$ for $q = 4n+3$ and
\item $N=\frac{q-5}{2}$ for $q = 4n+1$
\end{itemize}
\end{lemma}
\begin{proof}
 Note that \tref{thm:carlitz_Fp} gives the formula for calculating  $\gamma$ such that $x^{(q+1)/2} + \gamma x$ is a PP over $\mathbb{F}_q$. We ought to count the number of squares in $\mathbb{F}_q$ which are not equal to $0, \pm 1$. While $0, 1$ are squares in all $\mathbb{F}_q$, $-1$ is a square in those $\mathbb{F}_q$, where $q = 4n+1$. Thus, for $\mathbb{F}_q$, where $q = 4n+1$, the number of squares not equal to $0, \pm 1$ is $\frac{q-5}{2}$ (as every square has two square roots, except $0$ which has only one). For $q = 4n + 3$, as $-1$ is not a square, the number of squares not equal to $0,1$ is $\frac{q-3}{2}$. This now gives us $N$, the number of distinct $\gamma \in \mathbb{F}_q$ such that $x^{(q+1)/2} + \gamma x$ is a PP of $\mathbb{F}_q$. From \tref{thm:Fp_to_Fpsquare}, it is clear that the number of additional normalized PPs of the type $f(x) = (x^q + \delta_i x)^{(q+1)/2} + L(x)$ is $Nq^2(q-1)$. %$g^{(q-1)/2}=(-1)$ for all $g$ that are  primitive elements of $\mathbb{F}_q$. Therefore, the element $c=g^n$ for $4n=q-1$, has its square $c^2=g^{2n}=g^{(q-1)/2}=(-1)$. In summary, if the field $\mathbb{F}_q$ has cardinality $q$ of the form $4n+1$, there exists non zero $c$ such that $c^2=-1$. \\
 %In all fields, there exist additive inverses, such that $c^2=(-c)^2=1$. Thus, the count of elements $k=\frac{c^2-1}{c^2+1}$ which are useful for forming permutation binomial $f(x)=x^{(q+1)/2}+kx$; depends upon whether $q=4n+1$ or not. The element $c$ must be non zero, then, there are two elements $1$ and $-1$ for which $c^2=1$. Thus the minimum count of useful $k$-values  is $(q-3)/2$ for fields with cardinality $q \neq 4n+1$. Additionally, two more elements for which $c^2=-1$ are ruled out  for fields with cardinality $q=4n+1$, and the desired count of $k$-values is $(q-5)/2$.\\
 %\tref{thm:Fp_to_Fpsquare} connects this number of permutation binomials of this form over $\mathbb{F}_q$ with the required number of permutation polynomials of $\mathbb{}_{q^2}$ of the form $s^m+L(x)$. 
\end{proof}

We end this section by listing the counts of all normalized polynomials of the kind $s^m+L(x)$, for different values of index $m$; for small fields of even and odd characteristic. 

\begin{table}[!ht]
\begin{center}
\caption{Total Number of Normalized PPs of the shape  $f(x)=s^m+L(x)$ for small fields (counting for all possible choice of  $\delta_i$s and both rank $1$ and rank $2$ $L(x)$)}
\vspace{0.2cm}
\label{tab:Table3}
\begin{tabular}{|c|c|c|c|c|c|c|c|}
\hline
Field & $\mathbb{F}_{4^2}$ & $\mathbb{F}_{5^2}$ & $\mathbb{F}_{7^2}$ & $\mathbb{F}_{8^2}$  & $\mathbb{F}_{9^2}$  & $\mathbb{F}_{11^2}$ & $\mathbb{F}_{13^2}$ \\
\hline
$index~m=2$ & $\cellcolor{red!30}420$ & $480$ & $2016$  & $\cellcolor{red!30}7560$ & $5760$ & $13200$ & $26208 $ \\
\hline
$index~m=3$ & $180$ & $\cellcolor{red!30}1080$ & $2016$ & $\cellcolor{red!30}7560$ & $\cellcolor{green!30}38160$ & $\cellcolor{red!30}27720$ & $26208$\\
\hline
$index~m=4$ & $--$ & $480 $ & $\cellcolor{yellow!30}6720$ & $\cellcolor{red!30}7560$ & $5760$ & $13200$  & $26208$\\
\hline
$index~m=5$ & $--$ & $--$ & $\cellcolor{red!30}4368$ & $\cellcolor{red!30}7560$ & $\cellcolor{green!30}25200$ & $13200$  & $\cellcolor{red!30}54600$\\
\hline
$index~m=6$ & $--$ & $--$ & $2016$ & $\cellcolor{red!30}7560$ & $5760$ & $\cellcolor{yellow!30}71200$  & $26208$\\
\hline
$index~m=7$ & $--$ &$--$ & $--$ &  $3528$ & $\cellcolor{red!30}12240$ & $\cellcolor{red!30}27720$  & $\cellcolor{green!30}168168$ \\
\hline
$index~m=8$ & $--$ &$--$ & $--$ & $--$ & $5760$ & $13200$ & $26208$\\
\hline
$index~m=9$ & $--$ &$--$ &$--$  & $--$ & $--$ & $\cellcolor{red!30}27720$ & $26208$\\
\hline
$index~m=10$ & $--$ & $--$ & $--$ & $--$ & $--$ & $13200$ & $26208$  \\
\hline
$index~m=11$ & $--$ & $--$ & $--$ & $--$ & $--$ & $--$ & $\cellcolor{red!30}54600$\\
\hline
$index~m=12$ & $--$ &$--$ & $--$ & $--$ & $--$ & $--$ & $26208$\\
\hline
\end{tabular}
\end{center}
\end{table}
%%%%%%%%%%%%%%%%%%%%%%%%%%%%%%%%%%%%%%%%%%%%%%%%%%%
Note that the normalized PPs arising out of  \tref{thm:lowerbound} is applicable to every relevant box in Table~$3$ ($m < q$). The cases where \tref{thm:coprime} is also applicable are indicated in red. Cases where only \tref{thm:Fp_to_Fpsquare} is applicable in addition to \tref{thm:lowerbound} are indicated in  yellow. The boxes which are green represent cases where all three   \tref{thm:lowerbound}, \tref{thm:coprime} and \tref{thm:Fp_to_Fpsquare} are applicable.
%%%%%%%%%%%%%%%%%%%%%%%%%%%%%%%%%%%%%%%%%%%%%%%
\section{Compositional Inverses of PPs $f=s^m+L(x)$}
Polynomials of a given field, which have compositional inverse (as defined in Section $2$) are permutation polynomials.   Finding compositional inverses of permutation polynomials is not a straightforward task and one needs to employ some techniques like those provided in \cite{wangcomp2019}. However, as it turns out, if one can find the compositional inverse of a normalized PP, then one can utilize this compositional inverse to find the compositional inverses of every PP in the family associated with that normalized polynomial. Explicitly, if $f$ is a normalized PP with a compositional inverse $h$, then the compositional inverse of the PP $\alpha f + \beta$ (with $\alpha \in \mathbb{F}^*_{q^2}, \beta \in \mathbb{F}_{q^2}$) is given by $h(\frac{x-\beta}{\alpha})$. The compositional inverses of most  permutation polynomials that have the shape $f = s^m + L(x)$ where $s = x^q + \delta_ix$ with $\delta_i$ being a $(q+1)$-th root of unity and $L(x)$ being a linearized polynomial, turns out to have a similar shape. In fact, there is more granularity in the last statement, which is elaborated in this section. It turns out that all permutation polynomials that belong to the families of normalized PPs generated by \tref{thm:lowerbound} have compositional inverses that belong to families of normalized PPs that are also generated by \tref{thm:lowerbound}. Similarly, the PPs belonging to families of normalized PPs generated by \tref{thm:coprime} are closed under compositional inverses. On the other hand, PPs belonging to the families of normalized PPs generated by \tref{thm:Fp_to_Fpsquare} may have inverses whose shapes are different. We now describe how these compositional inverses can be determined.  

\begin{lemma}\label{lem:cat1inv}
Let $f(x) =(x^q+\delta_ix)^m+ L(x)$ be a PP of $\mathbb{F}_{q^2}$ such that $L(x)$ is a rank $2$ linearized polynomial that maps the kernel of $(x^q + \delta_ix)^m$ to the subspace containing the image of $(x^q + \delta_ix)^m$. Then $h(x)=\eta (x^q+\delta_j x)^m + M(x)$ is the compositional inverse of $f$, where the rank $2$ linearized polynomial $M$ is the compositional inverse of $L$.
\end{lemma}

\begin{proof}
Consider a basis $\{u,v\}$ of $\mathbb{F}_{q^2}$ where $u \in \ker (x^q + \delta_ix)$. Consider some $ \alpha u + \beta v = w \in \mathbb{F}_{q^2}$. If $s = (x^q + \delta_ix)$, then \[ f(w) = \beta^m s^m(v) +  L(w) = \beta^m z + L(w) \] where $z = s^m(v)$. Now $M(f(w))= \beta^m M(z) + w$, if $M$ is the compositional inverse of $L$. Consider a rank $1$ linearized polynomial $(x^q + \delta_j x)$, such that $L(u) \in \ker (x^q + \delta_jx)$. Note that $L(u), L(v)$ are linearly independent vectors in $\mathbb{F}_{q^2}$ and therefore form a basis for $\mathbb{F}_{q^2}$. Further note that $\alpha_1 L(u) = z$ for some $\alpha_1 \in \mathbb{F}_q$ (due to the construction specified in \tref{thm:lowerbound}). %Let $a = \alpha_1 L(u) + \beta_1 L(v)$. 

%At this point, there is a bifurcation in the process depending upon $f$. Case $1$ is when $\beta_1 = 0$ (this is the kind of construction prescribed in \tref{thm:lowerbound}). 
Let $b = L(v)^q + \delta_j L(v) \not= 0$. Find $\eta \in \mathbb{F}_{q^2}$  such that $\eta b^m = -M(z)$. Then, we claim that $h = \eta(x^q + \delta_jx)^m + M(x)$ is the compositional inverse of $f = (x^q + \delta_ix)^m + L(x)$. 

Note that if $s_1 = x^q + \delta_jx$, then $s_1(L(u)) = 0$ and $s_1(L(v)) = b$. %Further, from the nature of $f$, we also know that $a = \alpha_1 L(u)$ for some $\alpha_1 \in \mathbb{F}_q$. %Note that \begin{eqnarray*} f(w) & = & \beta^m a + \alpha L(u) + \beta L(v) \\ & = & (\beta^m\alpha_1 + \alpha)L(u) +  \beta L(v) \end{eqnarray*}
Therefore \begin{eqnarray*}h(f(w)) & = & \eta s_1(f(w))^m + M(f(w)) \\ & = & \eta s_1(\beta L(v))^m + \beta^m M(z) + \alpha u + \beta v \\ & = & \eta \beta^m s_1(L(v))^m + \beta^m M(z) + w \\ & = & \eta \beta^m b^m + \beta^m M(z) + w = w \end{eqnarray*}
Note that as $z = \alpha_1 L(u)$, therefore $M(z) = \alpha_1 u$ and therefore $M$ maps the kernel of $s_1$ to the subspace containing the image of $\eta (x^q + \delta_j x)^m$ and so $h = \eta (x^q + \delta_j x)^m + M(x)$ belongs to the set of PPs generated by \tref{thm:lowerbound}.
%Case 2 corresponds to $f$ constructed using \tref{thm:Fp_to_Fpsquare} where $a = \alpha_1 L(u) + \beta_1 L(v)$ with $\beta_1 \not= 0$. Let $b = a^q + \delta_j a \not= 0$. Find $\eta \in \mathbb{F}_{q^2}$ such that $\eta b^m = -\tilde{L}(a)$. In this case, $h = \eta(x^q + \delta_jx)^m + \tilde{L}(x)$ is the compositional inverse of $f$. \begin{eqnarray*} f(w) & = & \beta^m a + \alpha L(u) + \beta L(v) \\ & = & \beta^m a + \frac{\beta}{\beta_1}(\alpha_1 L(u) + \beta_1 L(v)) + (\alpha - \frac{\alpha_1 \beta}{\beta_1})L(u)\\  & = & \beta^m a + \frac{\beta}{\beta_1} a + (\alpha - \frac{\alpha_1 \beta}{\beta_1})L(u)
%\end{eqnarray*}
\end{proof}

Thus it is clear that the class of permutation polynomials that belong to the families associated to normalized PPs of the type $f(x) = s^m + L(x)$ generated as per the rules of \tref{thm:lowerbound} are closed under compositional inverses. In fact, these PPs can be divided into sub-classes associated to the value of the index $m$ in the range $2 \le m \le q-1$ and every one of these sub-classes are closed under compositional inverses.  \\
%We have found all permutation polynomials of this type and their compositional inverses over small fields-- $\mathbb{F}_{7^2}$, $\mathbb{F}_{11^2}$. Some cases in the notation of \lref{lem:cat1inv} are listed. We denote the primitive element with $a$.
\begin{example}
In this example we provide sample PPs for each index $m$ having the shape $f(x)=(x^q+\delta_i x)^m+L(x)$ (\tref{thm:lowerbound}) and their compositional inverses $h(x)=\eta(x^q+\delta_j x)^m+M(x)$ over $\mathbb{F}_{7^2}$. We denote by $a$ a primitive element of $\mathbb{F}_{7^2}$. \\
\begin{enumerate}
    \item $f(x)=(x^7+a^6x)^2+6ax^7+(2a+6)x$ and its compositional inverse is $h(x)=6a(x^7+a^{12}x)^2+ax^7+(5a+1)x$
    \item $f(x)=(x^7+a^6x)^3+(2a+5)x^7+(a+2)x$, with compositional inverse  \\$h(x)=6(x^7+a^6x)^3+6x^7+(2a+2)x$
    \item $f(x)=(x^7+a^{18}x)^4+3x^7+(a+5)x$,\\$h(x)=(6a+1)(x^7+a^6x)^4+6x^7+(2a+2)x$
    \item $f(x)=(x^7+a^{36}x)^5+4ax^7+(2a+4)x$, \\ $h(x)=3(x^7+a^{36}x)^5+ax^7+(6a+3)x$
    \item $f(x)=(x^7+a^{30}x)^6+(4a+5)x^7+(2a+4)x$, \\ $h(x)=(x^7+a^{36}x)^6+(4a+5)x^7+(2a+1)x$
\end{enumerate}
\end{example}
%%%%%%%%%%%%%%%%%%%%%%%%%%%%%%%%%%%%%%%%%%%%%%%%%%%

We now look at the compositional inverses of PPs that have a rank $1$ linearized polynomial $L(x)$ as obtained in \tref{thm:coprime}. Note that if $f(x) = s^m + L(x)$ (where both $s = x^q + \delta_i x$ and $L(x)$ are rank one linearized polynomials) is a PP, then by \tref{thm:coprime} $gcd(m,q-1) = 1$. Further, $x^m$ is a PP of $\mathbb{F}_q$. It is well known that the compositional inverse of the PP $x^m$ of $\mathbb{F}_q$ is $x^n$ where $mn \equiv 1 \mod (q-1)$.

\begin{lemma}\label{lem:cat2inv}
Let $f(x) = s^m + L(x)$ (where both $s = x^q + \delta_i x$ and $L(x)$ are rank one linearized polynomials) be a PP of $\mathbb{F}_{q^2}$. %$gcd(m,p-1)=1$ and let $L(x)$ be a rank $1$ linearized polynomial. If $f(x)=(x^p-\delta_i x)^m+\alpha L(x)$is a PP of $\mathbb{F}_{p^2}$, t
Then $h(x)=\eta (x^q+\delta_jx)^n + M(x)$ is the compositional inverse of $f(x)$, where $M(x)$ is a rank one linearized polynomial, $\eta \in \mathbb{F}_{q^2}$, $\delta_j$ a $(q+1)$-th root of unity and $n$ such that $mn \equiv 1~\mod~(q-1)$. 
\end{lemma}

\begin{proof}
Consider a basis $\{u,v\}$ of $\mathbb{F}_{q^2}$ such that $s(u) = 0$ and $L(v) = 0$ (one of the conditions in \tref{thm:coprime} ensures that $u,v$ are linearly independent). Consider $f = s^m + L(x)$ and a generic element of  $\mathbb{F}_{q^2} \ni w = \alpha u + \beta v$ where $\alpha, \beta \in \mathbb{F}_q$. \[ f(w) = s^m(w) + L(w) = \beta^m s^m(v) + \alpha L(u) \]
Let $s^m(v) = z$. Take $M$ to be a rank $1$ linearized polynomial such that $z \in \ker M$ and $M(L(u)) = u$. Take $s_j = x^q + \delta_j x$ such that $L(u) \in \ker s_j$. Find $\eta \in \mathbb{F}_{q^2}$ such that $\eta s_j(z)^n = v$ (here $n$ is such that $mn \equiv 1 \mod (q-1)$). We claim that $h(x) = \eta(x^q + \delta_j x)^n + M(x)$ is the compositional inverse of $f(x)$.
\begin{eqnarray*} h(f(w)) = h(\beta^m s^m(v) + \alpha L(u)) & = & h(\beta^m z + \alpha L(u)) \\ \mbox{\rm (considering defined kernels)} & = & \eta s_j^n(\beta^m z) + M(\alpha L(u)) \\ \mbox{\rm (considering the definitions)} & = & \beta^{mn} \eta s_j(z)^n + \alpha M(L(u)) \\ & = & \beta v + \alpha u = w
\end{eqnarray*}
\end{proof}

Clearly, all PPs belonging to the families of normalized PPs generated through application of \tref{thm:coprime} are closed under compositional inverses. Unlike the previous set of PPs generated through \tref{thm:lowerbound} which could be subdivided into sub-classes, each of which were closed under compositional inverses, the set of PPs generated through \tref{thm:coprime} does not have any easily describable sub-classes into which they can be divided. 
\begin{example}
We now provide examples of PPs over $\mathbb{F}_{11^2}$ of the form $f(x) = (x^q + \delta_ix)^m + L(x)$, where $L(x)$ is a rank 1 linearized polynomial and its inverse of the form $h(x) = \eta(x^q + \delta_jx)^n + M(x)$, where $M(x)$ is also a rank 1 linearized polynomial. Note that the indices $m$ and $n$ are such that $mn \equiv 1 \mod (q-1)$. Once again, $a$ represents a primitive element of $\mathbb{F}_{11^2}$.
\begin{enumerate}
   % \item $\mathbb{F}_{7^2}$: $f(x)=(x^7+a^{30}x)^5+(2a+2)x^7+(3a+4)x$ and $h(x)=(a+5)(x^7+a^{36}x)^5+(a+2)x^7+(6a+5)x$
    %\begin{itemize}
    \item $f(x)=(x^{11}+a^{50}x)^3+a^{30}(x^{11}+a^{70}x)$ and \\
    $h(x)=a^{60}(x^{11}+a^{50}x)^7+a^{110}(x^{11}+a^{30}x)$
    \item $f(x)=(x^{11}+a^{10}x)^7+a^{110}(x^{11}+a^{20}x)$ and \\
    $h(x)=a^{74}(x^{11}+a^{60}x)^3+a^{90}(x^{11}+a^{110}x)$
    \item $f(x)=(x^{11}+a^{30}x)^9+a^{51}(x^{11}+a^{110}x)$ and \\
    $h(x)=a^{35}(x^{11}+a^{100}x)^9+a^{48}(x^{11}+a^{30}x)$
   % \end{itemize}
\end{enumerate}
\end{example}
%%%%%%%%%%%%%%%%%%%%%%%%%%%%%%%%%%%%%%%%%%%%%%%%%%%%%%%%%%%%%%%%%%%%%%%
Unfortunately, this property of being closed under compositional inverses does not hold for PPs that belong to the families of normalized PPs arising out of \tref{thm:Fp_to_Fpsquare}. However, for odd $q$, there is a subclass of families of normalized PPs arising out of \tref{thm:Fp_to_Fpsquare} which are closed under compositional inverses. These are normalized PPs of the form $f(x) = (x^q + \delta_i x)^{(q+1)/2} + L(x)$where the index $m = \frac{q+1}{2}$.

There are some preliminaries required before the construction of inverses of this class of normalized PPs can be elaborated. Recall \tref{thm:carlitz_Fp} which states how to find $\gamma$ such that $x^{(q+1)/2} + \gamma x$ is a permutation polynomial of $\mathbb{F}_q$. It is clear from the formula given in \tref{thm:carlitz_Fp} that if $x^{(q+1)/2} +\gamma x$ is a PP of $\mathbb{F}_q$, then so is $x^{(q+1)/2} - \gamma x$ (replace $c$ in the formula with $1/c$). 

Further, any polynomial over $\mathbb{F}_q$ of the type $x^{(q+1)/2}+\gamma x$ which is a PP can be better understood by thinking of it as $x(x^{(q-1)/2} + \gamma)$. On $\mathbb{F}_q$, $x^{(q-1)/2}$ takes the value $1$ when $x$ is a square and takes the value $-1$ when $x$ is not a square. Thus the PP $x^{(q+1)/2}+\gamma x$ can be thought of as a permutation obtained by multiplying all the squares by $(\gamma +1)$ and all the nonsquares by $(\gamma -1)$. On any field, the product of two squares is a square, the product of two nonsquares is also a square, whereas the product of a square and a nonsquare is a nonsquare. So, if $x^{(q+1)/2}+\gamma x$ is a PP, then either $(\gamma +1), (\gamma -1)$ are both squares or they are both nonsquares.

\begin{lemma}\label{lem:comp_scalar}
Let $f(x) = \frac{x^m}{\gamma} + x$ be a PP on $\mathbb{F}_q$. Then $h(x) = x^m \mp \gamma x$ is such that $h(f(x)) = \mp \frac{\gamma^2-1}{\gamma}x$ with the choice $-$ or $+$ depending on whether  $\frac{1+\gamma}{\gamma}$ is a square or not.  
\end{lemma}

\begin{proof}
From the discussion above, it is clear that $f(x) = \frac{\gamma+1}{\gamma}x$ if $x$ is a square and $f(x) = \frac{\gamma -1}{\gamma}x$ if $x$ is not a square. As $f(x)$ is a PP, both  $\frac{\gamma +1}{\gamma}$ and $\frac{\gamma -1}{\gamma}$ are squares or both of them are not squares. If they are both squares,  then take $h(x) = x^m - \gamma x$. Then $h(f(x)) = h(\frac{\gamma + 1}{\gamma}x) = \frac{1 - \gamma^2}{\gamma}x $ for $x$ which is a square and $h(f(x)) = h(\frac{\gamma - 1}{\gamma}x) = (-1 - \gamma)\frac{\gamma - 1}{\gamma}x = \frac{1 - \gamma^2}{\gamma}x $ for $x$ which is not a square.

On the other hand, if $\frac{\gamma +1}{\gamma}$ and $\frac{\gamma -1}{\gamma}$ are both not squares, then take $h(x) = x^m + \gamma x$. For $x$ which is a square, $f(x) = \frac{\gamma + 1}{\gamma}x$ is not a square and therefore $h(f(x)) = h(\frac{\gamma + 1}{\gamma}x) = (\gamma - 1)\frac{\gamma + 1}{\gamma}x = \frac{\gamma^2 - 1}{\gamma}x$. Similarly, for $x$ which is not a square, $h(f(x)) = h(\frac{\gamma - 1}{\gamma}x) = (\gamma + 1)\frac{\gamma - 1}{\gamma}x = \frac{\gamma^2 - 1}{\gamma}x$.
\end{proof}
The above lemma is crucial to how the inverse permutation is constructed. For ease of notation, we use $s = x^q + \delta_i x$ and $m = \frac{q+1}{2}$.

\begin{theorem}\label{thm:Fp_to_Fosquare_comp_inv}
Let $f(x) = s^m + L(x)$ be a PP of $\mathbb{F}_{q^2}$ where $L(x)$ is a rank 2 linearized polynomial that does not map the $\ker s$ to ${\rm im}\,\, s^m$. Then $h(x) = \eta(x^q + \delta_j x)^m + M(x)$ is a compositional inverse of $f(x)$, where $M(x)$ is a rank 2 linearized polynomial that does not map $\ker (x^q + \delta_j x)$ to ${\rm im}\,\, (x^q + \delta_j x)^m$.
\end{theorem}

\begin{proof}
Let $u \in \ker s$ and $v \in L^{-1}({\rm im}\,\,s^m$) form a basis of $\mathbb{F}_{q^2}$. Let $s^m(v) = v_1$ and $L(v) = v_2$. Let $\gamma \in \mathbb{F}_q$ be such that $L(v)/s^m(v) = \gamma$ (since both $s^m(v)$ and $L(v)$ lie in the same one dimensional subspace which contains the image of $s^m$). Consider a generic element $w = \alpha u + \beta v$ where $\alpha, \beta \in \mathbb{F}_q$. Then \[f(w) = \beta^m v_1 + L(w) = \beta^m v_1 + \alpha L(u) + \beta L(v)  = \beta^m v_1 + \beta v_2 + \alpha L(u) \] and taking into consideration $\gamma$, we get, $f(w) = (\frac{\beta^m}{\gamma} + \beta) v_2 + \alpha L(u)$. 

Let $M(x)$ be a rank 2 linearized polynomial, such that $M(L(u)) = u$ and $M(L(v)) = M(v_2) = \frac{\gamma^2}{\gamma^2-1} v$.
Choose $\delta_j$ such that $L(u) \in \ker (x^q + \delta_j x)$. Let us assume $s_j = x^q + \delta_j x$. Therefore, $s_j(L(u)) = 0$. If $\frac{\gamma +1}{\gamma}$ is a square, then choose $k = \frac{-\gamma}{\gamma^2-1}$, else choose $k = \frac{\gamma}{\gamma^2-1}$. Let $\eta \in \mathbb{F}_{q^2}$ be such that $\eta s_j^m(v_2) = kv$.

Then we claim that  $h(x) = \eta s_j^m + M(x)$ is the compositional inverse of $f(x)$. To show this, we evaluate \begin{eqnarray*} h(f(w))  & =  & \eta s_j^m ((\frac{\beta^m}{\gamma} + \beta ) v_2) + (\frac{\beta^m}{\gamma} + \beta)M(v_2) + \alpha M(L(u)) \\ & = & \left[ (\frac{\beta^m}{\gamma} + \beta)^m k + (\frac{\beta^m}{\gamma} + \beta)\frac{\gamma^2}{\gamma^2-1} \right] v + \alpha u  \\ & = & \frac{\mp \gamma}{\gamma^2 - 1}\left[ (\frac{\beta^m}{\gamma} + \beta)^m \mp \gamma (\frac{\beta^m}{\gamma} + \beta) \right] v + \alpha u \\ \mbox{\rm applying \, \lref{lem:comp_scalar} } & = & \frac{\mp\gamma}{\gamma^2-1} \left[ \frac{\gamma^2-1}{\mp \gamma}\beta \right]v + \alpha u = \alpha u + \beta v = w \end{eqnarray*} 
\end{proof}

Clearly, the inverse PP $h(x)$ belongs in the family $(x^q + \delta_jx)^{(q+1)/2} + \frac{1}{\eta} M(x)$ and so the PPs arising out of normalized PPs of the form $(x^q + \delta_i x)^{(q+1)/2} + L(x)$ where $\delta_i$ is a $(q+1)$-th root of unity are closed under compositional inverses.
Observe that unlike the inverses constructed in \lref{lem:cat1inv}, here the linearized polynomials $L(x)$ and $M(x)$ are not inverses of each other. In fact the map $M(L(x))$ is a map with eigenvalues $1,\frac{\gamma^2}{\gamma^2-1}$ and corresponding eigenvectors $u,v$. Similarly, for the map $L(M(x))$, the eigenvectors are $L(u),L(v)$ corresponding to the eigenvalues $1,\frac{\gamma^2}{\gamma^2-1}$ respectively. %We give a few examples of these PPs arising from the permutation binomials of degree $(q+1)/2$ of the base field, for $\mathbb{F}_{13}$.
\begin{example}
Consider the field $\mathbb{F}_{169}$: $f(x)=(x^{13}+a^{12}x)^7+a^{87}(x^{13}+a^{107}x)$ is a permutation polynomial of $\mathbb{F}_{13^2}$ which is of the form $s^7 + L(x)$ where $L(x)$ does not map the kernel of $s$ to the image of $s^7$. This permutation polynomial can be associated to the permutation binomial $x^7+2x$ of the base field $\mathbb{F}_{13}$. Its compositional inverse is $h(x)=a^{50}(x^{13}+a^{168}x)^7+a^{59}(x^{13}+a^{167}x)$, which is another PP of the form $\eta s^7 + M(x)$. 
%Another permutation binomial in $\mathbb{F}_{13}[x]$ is $x^7+6x$. We have a PP $f(x)=(x^{13}+a^{12}x)^7+a^{21}(x^{13}+a^{61}x)$ and the compositional inverse of this PP is $h(x)=a^{120}(x^{13}+a^{168}x)^7+a^{121}(x^{13}+a^{147}x)$. In this case $L(M(x))=a^{122}(x^{13}+a^{54}x)$ and $M(L(x))=a^{98}(x^{13}+a^6x)$. 
\end{example}
%%%%%%%%%%%%%%%%%%%%%%%%%%%%%%%%%%%%%%%%%%%%%%%%%%%%%%%%%%%%%%%
\section{Constructing PPs of $\mathbb{F}_{q^2}$ from s-Polynomials}

So far we have been considering rank $1$ linearized polynomials $s = x^q + \delta_ix$ and constructing PPs by adding linearized polynomials $L(x)$ to some monomial in $s$, say $s^m$. A natural extension would be to consider polynomials in $s$. By arguments we have made earlier, it is enough to consider polynomials of the form $g(s) = s^m + \Sigma_{i=2}^{m-1} a_i s^i$, with $m < q$ and $a_i \in \mathbb{F}_{q^2}$, i.e., $g \in \mathbb{F}_{q^2}[x]$ with $\deg g < q$. We shall refer to these polynomials as s-polynomials.

In the case of $s^m$, the image of this map always lies within a one dimensional subspace of $\mathbb{F}_{q^2}$. That makes a lot of the results that followed, easy to deduce. Unfortunately, when we look at more general s-polynomials, this property of the image lying within a one dimensional subspace of $\mathbb{F}_{q^2}$ no longer holds. Therefore, as a first step, we shall restrict ourselves to those s-polynomials $g(s)$ such that ${\rm im}\, g(s)$ lies within a one dimensional subspace of $\mathbb{F}_{q^2}$.

For this purpose, we shall choose a special case of rank 1 linearized polynomial where $\delta_i = 1$, that is, the polynomial $s = x^q + x$. This is in fact the trace polynomial and it is well known that the trace function maps elements of $\mathbb{F}_{q^2}$ to $\mathbb{F}_q$. Since $\mathbb{F}_q$ is a field, therefore any polynomial $g(s) = s^m + \Sigma_{i=2}^{m-1} a_i s^i$ with $a_i \in \mathbb{F}_q$ would ensure that the image of $g(s)$ also lies within $\mathbb{F}_{q}$. In other words, we first consider $g \in \mathbb{F}_{q}[x]$ with $\deg g < q$ and $s = x^q + x$. We immediately get a generalization of \tref{thm:lowerbound} that we now state.
\begin{theorem}\label{thm:pp_from_basefield_poly}
Let $s = x^q + x$ and $g(s) = s^m + \Sigma_{i=2}^{m-1} a_i s^i$, where $a_i \in \mathbb{F}_q$. Then any polynomial $f = g(s) + L(x)$ is a permutation polynomial of $\mathbb{F}_{q^2}$ if $L(x)$ is a rank 2 linearized polynomial of $\mathbb{F}_{q^2}$ that maps the kernel of $s$ to $\mathbb{F}_q$.
\end{theorem}
\begin{proof}
This proof runs along the lines of the proof of \tref{thm:lowerbound}. Let $u$ belong to the kernel of $x^q + x$ and $v \in \mathbb{F}_{q^2}$, such that $\{u,v\}$ form a basis of $\mathbb{F}_{q^2}$. For any $\alpha_i u + \beta_i v = w_i \in \mathbb{F}_{q^2}$ where $\alpha_i, \beta_i \in \mathbb{F}_q$, we have $f(w_i) = g(\beta_i s(v)) + L(w_i)$. We now demonstrate that for all $w_i \not= w_j$, $f(w_i) \not= f(w_j)$, which would prove that $f$ is a PP.

If $w_i,w_j \in \mathbb{F}_{q^2}$ such that $\beta_i = \beta_j$ and $\alpha_i \not= \alpha_j$, then as $L(x)$ is a rank 2 linearized polynomial, $L(w_i) \not= L(w_j)$ and therefore $f(w_i) \not= f(w_j)$. 

On the other hand, if $\beta_i \not= \beta_j$ and we assume that $f(w_i) = f(w_j)$, then a rearrangement of terms give us \[g(\beta_i s(v)) - g(\beta_j s(v)) + (\alpha_i-\alpha_j)L(u) = (\beta_j-\beta_j)L(v)      \]
Notice that the quantity on the left hand side is in $\mathbb{F}_q$ whereas, by virtue of $L(x)$ being a rank 2 linearized polynomial that maps the kernel of $x^q + x$ to $\mathbb{F}_q$, $L(v) \not\in \mathbb{F}_q$. This contradicts the assumption that $f(w_i) = f(w_j)$.
\end{proof}

On the same lines as the counting done in \cref{cor:lowerbound}, one can conclude that there are at least $q(q-1)^2$ PPs of the kind $g(s)+L(x)$ for every s-polynomial $g \in \mathbb{F}_{q}[x]$ when $s$ is the trace function. For permutation polynomials arising from the construction given in \tref{thm:pp_from_basefield_poly}, we now demonstrate how to construct the compositional inverse.

\begin{lemma}\label{lem:pp_base_inv}
Let $f(x) = g(s) + L(x)$ be a permutation polynomial of $\mathbb{F}_{q^2}$, where $s = x^q + x$ and $g(s) = s^m + \Sigma_{i=2}^{m-1}a_is^i$, where $a_i \in \mathbb{F}_q$. Further, $L(x)$ is a rank 2https://www.overleaf.com/project/61c16380bf92059009c01e18/detacher linearized polynomial that maps the kernel of $x^q + x$ to $\mathbb{F}_q$. Then $h(x) = g_1(x^q - x) + M(x)$ is the compositional inverse of $f$, where $M(x)$ is a rank 2 linearized polynomial which is the compositional inverse of $L(x)$ and the $s$-polynomial $g_1(s_1) = \Sigma_{i=2}^{m}\eta_is_1^i$ has the same degree as the $s$-polynomial $g(s)$.  
\end{lemma}

\begin{proof}
Let $u \in \ker (x^q + x)$ and $v \in \mathbb{F}_{q^2}$ form a basis of $\mathbb{F}_{q^2}$, where $v^q + v = 1$. Then for any $w = \alpha u + \beta v$ with $\alpha,\beta \in \mathbb{F}_q$, we have \[ f(w) = g(\beta s(v)) + L(w) = g(\beta) + L(w) \] Let $M$ be the compositional inverse of the linearized polynomial $L$. Then $M(f(w)) = M(g(\beta)) + w$. Consider the rank 1 linearized polynomial $x^q - x$. Note that $\ker (x^q - x) = \mathbb{F}_q$. The action of $x^q - x$ on $f(w)$ yields \[  f(w)^q - f(w) = \beta (L(v)^q - L(v)) = \beta z \] where $z = L(v)^q - L(v)$. Note that the above simplification is possible since $g(\beta), \alpha L(u) \in \mathbb{F}_q$ and $\beta^q = \beta$. Now, $g(\beta)  = \beta^m + \Sigma_{i=2}^{m-1}a_i\beta^i$ and therefore $M(g(\beta)) = M(\beta^m) + \Sigma_{i=2}^{m-1}a_iM(\beta^i) = g(\beta)M(1)$. Choose $\eta_i \in \mathbb{F}_{q^2}$ such that $\eta_i z^i = -a_iM(1)$ for $i = 2, \cdots, m$ (take $a_m = 1$). Let $g_1(x^q-x) = \eta_m (x^q - x)^m + \Sigma_{i=2}^{m-1}\eta_i(x^q-x)^i$. 

Then we claim that $h(x) = g_1(x^q - x) + M(x)$ is the compositional inverse of $f(x)$. \begin{eqnarray*} \left[ \Sigma_{i=2}^m \eta_i (x^q-x)^i + M(x) \right](f(w)) & = & \Sigma_{i=2}^m \eta_i (\beta z)^i + g(\beta)M(1) + w \\ & = & \Sigma_{i=2}^m \beta^i(\eta_i z^i + a_iM(1)) + w = w \end{eqnarray*}  
\end{proof}

So far, we have only dealt with normalized PPs arising out of the special rank 1 polynomial which is the trace function.  What about the other rank one linearized polynomials $s = x^q + \delta_ix$? As it turns out, there is a quick generalization of \tref{thm:pp_from_basefield_poly}.

\begin{corollary}\label{cor:pp_from_base}
Let $s = x^q + \delta_ix$ be a rank 1 linearized polynomial. Then every s-polynomial of the kind $g(s) = s^m + \Sigma_{i=2}^{m-1} \lambda^{m-i}a_is^i$ (where $a_i \in \mathbb{F}_q$ and $\lambda$ is in the image of $s$) along with every rank 2 linearized polynomial $L(x)$ that maps the kernel of $s$ to the one dimensional space spanned by $\lambda^m$ forms permutation polynomials of $\mathbb{F}_{q^2}$ of the shape $f = g(s) + L(x)$. 
\end{corollary}
\begin{proof}
Consider $\lambda$ in the image of $s$ and let $\eta = \lambda^{-1}$. Then the image of the polynomial  $\eta s = \eta(x^q + \delta_i x)$ is $\mathbb{F}_q$. Take any s-polynomial $g(s) = s^m + \Sigma_{i=2}^{m-1} a_is^i$ with $a_i \in \mathbb{F}_q$. Then arguing along the lines of \tref{thm:pp_from_basefield_poly}, one can conclude that every rank 2 linearized polynomial $L(x)$ that maps the kernel of $s$ to $\mathbb{F}_q$ would form a PP of the shape $f = g(\eta s) + L(x)$. These PPs are not monic but have a leading coefficient $\eta^m$ and making them monic would modify the s-polynomial to $s^m + \Sigma_{i=2}^{m-1} \eta^{i-m}a_is^i$ which is of the form given in the statement of this corollary as $\lambda = \eta^{-1}$. The linearized polynomial $L(x)$ gets modified to $\eta^{-m}L(x)$. Thus the linear map $\eta^{-m} L(x)$ maps the kernel of $s$ to the span of $\eta^{-m} = \lambda^m$.
\end{proof}

Of course, the $\lambda$ in \cref{cor:pp_from_base} is not unique, but all the other $\lambda$ in the image of $s$ are of the form $\alpha_i\lambda$ with $\alpha_i \in \mathbb{F}_q$. So it is enough to consider any one $\lambda$ in the image of each rank 1 linearized polynomial $x^q + \delta_ix$. The compositional inverses of PPs arising from \cref{cor:pp_from_base} can be obtained with a slight tweak of \lref{lem:pp_base_inv}.

\begin{corollary}\label{cor:gen_s}
Let $f = g(s) + L(x)$ be a normalized PP of $\mathbb{F}_{q^2}$, where $s = x^q + \delta_ix$ is a rank 1 linearized polynomial, $g(s) = \Sigma_{i=2}^m \lambda^{m-i}a_is^i$ (where $a_i \in \mathbb{F}_q$, $a_m = 1$ and $\lambda$ is a nonzero element in the image of $s$) and $L(x)$ is a rank 2 linearized polynomial that maps the kernel of $s$  to the one dimensional space spanned by $\lambda^m$. Then $h(x) = g_1(x^q + \delta_jx) + M(x)$ is the compositional inverse of $f$, where $M(x)$ is a rank 2 linearized polynomial which is the compositional inverse of $L(x)$ and $\lambda^m \in \ker (x^q + \delta_jx)$. 
\end{corollary}

\begin{proof}
The proof is the same as that of \lref{lem:pp_base_inv} but with the construction made a bit more explicit. We begin by re-writing $g(s) = \Sigma_{i=2}^m \lambda^{m-i}a_is^i = \lambda^m \Sigma_{i=2}^m a_i\lambda^{-i}s^i = \lambda^m g'(\lambda^{-1}s)$ where $g'(s) = \Sigma_{i=2}^m a_is^i$. Let $\{u,v\}$ form a basis for $\mathbb{F}_{q^2}$ where $u \in \ker (x^q + \delta_ix)$. Let $v_1 = \lambda^{-1}s(v) = \lambda^{-1}(v^q + \delta_iv) \in \mathbb{F}_q$. For any $w = \alpha u + \beta v$ with $\alpha,\beta \in \mathbb{F}_q$, we have \[ f(w) = g(\beta s(v)) + L(w) = \lambda^m g'(\beta v_1) + L(w) \] If $M$ is the compositional inverse of $L$, then $M(f(w)) = M(\lambda^mg'(\beta v_1)) + w$. Let $s_1 = x^q + \delta_jx$ be such that $s_1(\lambda^m) = 0$. Then \[ s_1(f(w)) = s_1(\lambda^mg'(\beta v_1)) + \alpha s_1(L(u)) + \beta s_1(L(v)) = \beta v_2 \] where $v_2 = s_1(L(v))$. Note that $s_1(L(u)) = 0$ and $s_1(\lambda^mg'(\beta v_1)) = 0$. Find $\eta_i \in \mathbb{F}_{q^2}$ such that $\eta_i v_2^i = -a_iM(\lambda^{m} v_1^i)$ for $i=2,\cdots,m$. Let $g_1(s_1) = \Sigma_{i=2}^m \eta_is_1^i$. Then $h(x) = g_1(s_1) + M(x)$ is the compositional inverse of $f(x)$ which can be verified by evaluating $h(f(w)) = w$.
\end{proof}

Observe from the above constructions that the $s$-polynomials $g(s), g_1(s)$ contain the same monomials. Thus if the $s$-polynomial that generates the PP is a trinomial, then the compositional inverse also gets generated from a $s$-polynomial which is a  trinomial consisting of the same monomials. Of course, this property holds for PPs generated by \tref{thm:pp_from_basefield_poly} and \cref{cor:pp_from_base} where the appropriate rank 1 linearized polynomial $s = x^q + \delta_ix$ is to be used in the $s$-polynomials. 

There are several papers in literature that deal with PPs of a form quite close to what we have been discussing in \tref{thm:pp_from_basefield_poly} and \cref{cor:pp_from_base}. Most of those papers have too specific a structure imposed on them. Some of these papers include  \cite{Yuan2008PermutationPO}, \cite{ZHA2012781}, \cite{Zheng2015}, \cite{Zheng:2016:LCP:2995737.2995773}, \cite{Wu_Yuan}. Some of these papers use the trace polynomial, some of them restrict themselves to fields with characteristic 2 or 3 and so on. We claim PPs of the form $f(x) = g(s) + L(x)$ that appear in those papers are but a small fraction of the total number of PPs of this form.  %\textcolor{red}{Some of these papers which describe permutation polynomials using such trace functions, over fields of characteristic $2$ and other}

Note that the kernel of a rank $1$ linearized polynomial $s = x^q + \delta_ix$ is a one dimensional subspace of $\mathbb{F}_{q^2}$. In the same way, if one considers $x^q + \delta_ix + \lambda$, where $\lambda$ is in the image of $s = x^q + \delta_ix$, then the kernel of this function is a one dimensional affine line in $\mathbb{F}_{q^2}$. Thus, application of $s^m$ which was the $s$-polynomial used in Section~3, with $s = x^q + \delta_ix + \lambda$ results in $(x^q + \delta_ix + \lambda)^m = s^m + \Sigma_{i=0}^{m-1}\begin{pmatrix} m \\ i \end{pmatrix} \lambda^{m-i}s^i$. These translate to very specific subset of $s$-polynomials that have appeared in \tref{thm:pp_from_basefield_poly} and \cref{cor:pp_from_base}. Note that as far as $s$-polynomials are concerned, the terms corresponding to $s^0$ and $s^1$ are dropped and the normalized PPs generated using this particular  $s$-polynomial would contain polynomials of the form $(x^q + \delta_ix + \lambda)^m + L(x)$ in their associated family of PPs. We mention this here, since PPs of the form $(x^q + \delta_ix + \lambda)^m + L(x)$ have been often studied in literature but these form only a small subset of the family of PPs that appear associated to the normalized PPs appearing in \tref{thm:pp_from_basefield_poly} and \cref{cor:pp_from_base}.

The total number of s-polynomials of the form $g(s) = s^m + \Sigma_{i=2}^{m-1} a_is^i$ with $a_i \in \mathbb{F}_q$ and $2 \le m < q$ is equal to $1+q+\cdots + q^{q-3} = \frac{q^{q-2}-1}{q-1}$. Note that this count includes the s-polynomials that are pure monomials. Thus the trace function can be used as $s$ for each one of these $s$-polynomials and every one of these s-polynomials would generate at least $q(q-1)^2$ normalized PPs by \tref{thm:pp_from_basefield_poly}. By \cref{cor:pp_from_base}, each one of these s-polynomials give another $q$ other $s$-polynomials corresponding to each of the other rank one linear maps that may be used as $s$ and every one of these s-polynomials in turn generate at least $q(q-1)^2$ normalized PPs when added to appropriate rank 2 linearized polynomials. So the net total of normalized PPs are $q(q^2-1)(q^{q-2}-1)$. If one also counts the PPs belonging to the families associated to each normalized PP, then this number stands at $q^3(q^2-1)^2(q^{q-2}-1)$. Note that this number is actually a lower bound to the actual number of PPs that can be generated of the kind $g(s) + L(x)$ and so we now explore for some more normalized PPs of this shape.

\begin{theorem}\label{thm:from_perm}
Let $g(x) = x^m + \Sigma_{i=2}^{m-1} a_ix^i$ with $a_i \in \mathbb{F}_q$. If $g(x) + \gamma x$ is a PP of $\mathbb{F}_q$ for some $\gamma \in \mathbb{F}_q^*$, then there exist normalized PPs of the form $s^m + \Sigma_{i=2}^{m-1} \lambda^{m-i}a_is^i + L(x)$ for appropriate rank 1 linearized polynomial $s = x^q + \delta_ix$, $\lambda$ belonging to the image of $s$ and  where $L(x)$ is a rank $2$ linearized polynomial that does not map the kernel of $s$ to the span of $\lambda^m$.

If there are $k$ distinct $\gamma \in \mathbb{F}_q$ such that $g(x) + \gamma x$ is a PP of $\mathbb{F}_q$, then there are $kq^2(q-1)$ distinct rank $2$ linearized polynomials $L(x)$ that do not map the kernel of $s$ to the span of $\lambda^m$, such that $s^m +\Sigma_{i=2}^{m-1} \lambda^{m-i}a_is^i + L(x)$ are normalized PPs of $\mathbb{F}_{q^2}$.
 \end{theorem}

\begin{proof}
It is enough to consider $s= x^q+x$ and prove that there exists normalized PPs of the form $g(s)+L(x)$ where $L(x)$ is a rank 2 linearized polynomial that does not map the kernel of $s$ to $\mathbb{F}_q$. By arguments similar to the ones in \cref{cor:pp_from_base}, the above theorem would then follow.

Consider a linear map $L(x)$ that does not map the kernel of $s$ to $\mathbb{F}_q$. Consider a basis $\{u,v\}$ of $\mathbb{F}_{q^2}$ where  $u \in \ker s$ and $v = L^{-1}(1)$. Let $s(v) = z$ for some $z \in \mathbb{F}_q$. We claim that $f = g(s) + \gamma zL(x)$ is a PP of $\mathbb{F}_{q^2}$, when $g(x) + \gamma x$ is a PP of $\mathbb{F}_q$.

It is enough to show that for distinct $w_1,w_2 \in \mathbb{F}_{q^2}$, $f(w_1) \not= f(w_2)$. Let $w_i = \alpha_i u + \beta_i v$ with $\alpha_i,\beta_i \in \mathbb{F}_q$.  Then \begin{eqnarray*} f(w_i) = g(s)(w_i) + \gamma zL(w_i) & = & g(\beta_is(v)) + \gamma zL(w_i) \\ & = & g(\beta_i z) + \gamma z \alpha_i L(u) + \gamma z\beta_i \end{eqnarray*} %Consider the case where $\beta_1 = \beta_2$. Since $g(\beta_1a) = g(\beta_2a)$ and $L(w_1) \not= L(w_2)$, therefore $f(w_1) \not= f(w_2)$. 
%Suppose $\beta_1 \not= \beta_2$
Since $g(x) + \gamma x$ is a PP of $\mathbb{F}_q$, therefore $g(\beta_iz) + \gamma\beta_iz \in \mathbb{F}_q$ are distinct for distinct $\beta_i$. As $L(u) \not\in \mathbb{F}_q$, and both $\gamma, z \in \mathbb{F}_q$, therefore the only way $f(w_1) = f(w_2)$ is if $\alpha_1 = \alpha_2$ and $\beta_1 = \beta_2$, that is $w_1 = w_2$.

In order to count the number of distinct $L(x)$ that can be added to a particular s-polynomial to give normalized PPs of $\mathbb{F}_{q^2}$, let us consider monic $L(x)$. There are $q(q-1)$ monic rank 2 linearized polynomials. Each such $L(x)$ can be multiplied by $\eta \in \mathbb{F}_{q^2}^*$ to give other rank 2 linear maps. If $\eta \in \mathbb{F}_q$, then following the construction given above, one gets back the normalized PP that one already has counted (in other words, the counting is to be done using projective geometry arguments). So it is enough to consider only one $\eta$ from each line in $\mathbb{F}_{q^2}$.  There is a line of $q-1$ elements that map $L(u)$ to $\mathbb{F}_q$ and therefore $\eta$ corresponding to this line is not to be counted. Each of the other $q$ lines yields as many PPs as the number of $\gamma$ that make $g(x) + \gamma x$ a PP of $\mathbb{F}_q$. So there are $q(q-1)$ monic $L(x)$,  each yielding $k$ PPs from each of the $q$ appropriate lines, totalling to $kq^2(q-1)$. 
\end{proof}

We now demonstrate how to find the compositional inverses of PPs generated using \tref{thm:from_perm}. 

\begin{theorem}
Let $g(x) = x^m + \Sigma_{i=2}^{m-1} a_ix^i$ with $a_i \in \mathbb{F}_q$ be such that $g(x) + \gamma x$ is a PP of $\mathbb{F}_q$ for some $\gamma \in \mathbb{F}_q^*$. The compositional inverse of normalized PPs of the form $s^m + \Sigma_{i=2}^{m-1} \lambda^{m-i}a_is^i + L(x)$ for appropriate rank 1 linearized polynomial $s = x^q + \delta_ix$, $\lambda$ belonging to the image of $s$ and  where $L(x)$ is a rank $2$ linearized polynomial that does not map the kernel of $s$ to the span of $\lambda^m$ is of the form $h(x) = g_1(x^q + \delta_jx) + M(x)$ where $M(x)$ is a rank $2$ linearized polynomial that does not map the kernel of $x^q + \delta_jx$ to the span of image $g_1(x^q + \delta_jx)$.

\end{theorem}

\begin{proof}
%We shall demonstrate the proof for the special case when $s = x^q + x$ is the trace function. Let $f(x) = g(s) + L(x)$ be a PP of $\mathbb{F}_{q^2}$ of the kind mentioned in the statement of this theorem with $\lambda = 1$. Consider a basis of $\mathbb{F}_{q^2}$, $\{u,v\}$ where $u \in \ker s$ and $v = L^{-1}(1)$. A generic element of $\mathbb{F}_{q^2}, w = \alpha u + \beta v$ with $\alpha, \beta \in \mathbb{F}_q$, then maps to $f(w) = g(\beta s(v)) + L(w) = g(\beta z) + \alpha L(u) + \beta$, where $z = s(v) = v^q + v \in \mathbb{F}_q$. 
We can rewrite the normalized PP $s^m + \Sigma_{i=2}^{m-1} \lambda^{m-i}a_is^i + L(x)$ as $f(x) = \lambda^m g(\lambda^{-1}s) + L(x)$. Consider a basis of $\mathbb{F}_{q^2}$, $\{u,v\}$ where $u \in \ker s$ and $v = L^{-1}(\lambda^m)$. A generic element of $\mathbb{F}_{q^2} \ni w = \alpha u + \beta v$ with $\alpha, \beta \in \mathbb{F}_q$, then maps to $f(w) = \lambda^mg(\lambda^{-1}\beta s(v)) + L(w) = \lambda^mg(\beta z) + \alpha L(u) + \lambda^m\beta$, where $z = \lambda^{-1}s(v) \in \mathbb{F}_q$, as $\lambda \in {\rm im}\,s$.

Observe the action of the PP $f(x) = \lambda^mg(\lambda^{-1}s) + L(x)$ on two specific one dimensional subspaces of $\mathbb{F}_{q^2}$. The subspace spanned by $u$ gets mapped to the subspace spanned by $L(u)$, that is, $f(\alpha u) = \alpha L(u)$ for $\alpha \in \mathbb{F}_q$. The other subspace of interest is the subspace spanned by $v$. Observe that $f(v) = \lambda^mg(\lambda^{-1}s(v)) + L(v) = \lambda^mg(z) + \lambda^m$ and $f(\beta v) = \lambda^m(g(\beta z) + \beta)$ that lies in the span of $\lambda^m$. Thus this subspace spanned by $v$ maps onto the one dimensional subspace spanned by  $\lambda^m$. Let $\gamma = 1/z$. Then $f(\beta v) = \lambda^m(g(\beta z) + \gamma \beta z) = \lambda^m\phi(\beta z)$ where $\phi(x) = g(x) + \gamma x$. As $f(x)$ is a PP of $\mathbb{F}_{q^2}$ and the $q-1$ non-zero elements in the one dimensional space spanned by $v$ map bijectively to non-zero elements in the span of $\lambda^m$, therefore $\phi(x)$ must be a PP of $\mathbb{F}_q$.  Let $\psi(x) = \Sigma_{i=1}^n b_i x^i$ be the compositional inverse of $\phi(x)$ and therefore a PP of $\mathbb{F}_q$.

Consider a rank $2$ linearized polynomial $M$, such that $M(L(u)) = u$ and $M(L(v)) = M(\lambda^m) = \gamma b_1 v$, where $b_1 \in \mathbb{F}_q$ is the coefficient of the linear term in $\psi(x)$. Then \begin{eqnarray*} M(f(w)) = M(\lambda^m(g(\beta z) + \beta)) + M(\alpha L(u)) & = & M(\lambda^m\phi(\beta z)) + \alpha M(L(u)) \\ & = & \phi(\beta z)\gamma b_1 v + \alpha u \end{eqnarray*} Choose the rank $1$ monic polynomial $s_j = x^q + \delta_jx$ such that $L(u) \in \ker s_j$. Construct the polynomial $g_1(x) = \Sigma_{i=2}^n c_i x^i$ using the following relation $c_i = \gamma b_i v/s_j(\lambda^m)^i$ where $b_i$ are the coefficients of $\psi(x)$. Then we claim the polynomial $h(x) = g_1(x^q + \delta_jx) + M(x)$ is the compositional inverse of $f(x)$.

We have already seen that for $w = \alpha u + \beta v$, $f(w) = \lambda^m(g(\beta z) + \beta) + \alpha L(u) = \lambda^m\phi(\beta z) + \alpha L(u)$. Therefore, \begin{eqnarray*}
h(f(w)) = g_1(s_j(f(w)) + M(f(w)) & = & g_1(s_j(\lambda^m\phi(\beta z))) + \gamma b_1 \phi(\beta z)v + \alpha u \\ & = & g_1(\phi(\beta z)s_j(\lambda^m)) + \gamma b_1 \phi(\beta z)v + \alpha u \\ & = & \Sigma_{i=2}^n c_i \phi(\beta z)^i s_j(\lambda^m)^i + \gamma b_1 \phi(\beta z)v + \alpha u \\ \mbox{\rm (putting in the values of $c_i$}) & = & \Sigma_{i=2}^n \gamma b_i \phi(\beta z)^iv + \gamma b_1 \phi(\beta z)v + \alpha u \\ \mbox{\rm (by the definition of $\psi$)}& = & \gamma \psi(\phi(\beta z)) v + \alpha u \\ \mbox{\rm (as $\gamma z = 1$) } & = & \gamma \beta z v + \alpha u \\ & = & \beta v + \alpha u = w \end{eqnarray*} This proves that $h(x)$ is the compositional inverse of $f(x)$. Note that the rank $2$ map $M(x)$ is such that the $\ker s_j$ (which is the one dimensional subspace containing $L(u)$) does not map to the span of image of $g_1(s_j)$ (which is contained within the one dimensional subspace spanned by $v$). 
\end{proof}
%%%%%%%%%%%%%%%%%%%%%%%%%%%%%%%%%%%%%%%%%%%%%%%%%%%%%%%%%%%5
Observe here that the linear maps that appear in $f(x)$ and its inverse, namely, the maps $L(x)$ and $M(x)$ are not inverses of one another. In fact, the two eigenvalues of the map $M(L(x))$ are $1, \gamma b_1$ (here $\gamma$ and $b_1$ are as defined in the proof above) with corresponding eigenvectors being $u,v$. Similarly, the eigenvalues of $L(M(x))$ are also $1,\gamma b_1$ and the corresponding eigenvectors are $L(u), L(v)$.
\begin{example}
There are $2801$ monic $s$-polynomials (including the monomials) over $\mathbb{F}_{7^2}$ whose coefficients are in $\mathbb{F}_7$. By \tref{thm:pp_from_basefield_poly} each of these $s$-polynomials, using the trace function, can be associated with at least $252$ different rank 2 linear maps to generate PPs. \tref{thm:from_perm} states that those $s$-polynomials arising out of PPs of $\mathbb{F}_7$ where a linear term has been removed, can generate extra PPs for $\mathbb{F}_{7^2}$. There are precisely $120$ normalized PPs on $\mathbb{F}_7$ which includes the PP $x$ (permutation corresponding to the identity permutation). Among the other $119$ normalized PPs of $\mathbb{F}_7$, precisely $15$ do not contain the linear term. As we already saw in \egref{ex:add_count}, $x^4 + 3x$ and $x^4 + 4x$ are two of those $104$ normalized PPs of $\mathbb{F}_7$ with linear terms. Both of these PPs induce the $s$-polynomial $s^4$. Therefore this $s$-polynomial has $2 \times 294 = 588$ extra PPs associated to it, making the total count $252 + 588 = 840$. The rest $102$ normalized PPs of $\mathbb{F}_7$ that have linear terms induce unique $s$-polynomials and so these $s$-polynomials have an additional $294$ PPs each, making their count of PPs $252 + 294 = 546$. Thus, out of a total of $2801$ monic $s$-polynomials from $\mathbb{F}_{7}[x]$, all utilizing the trace function for $s$, precisely one would generate $840$ normalized PPs, $102$ $s$-polynomials would generate $546$ normalized PPs each and the rest would generate $252$ normalized PPs. This is a total of $736074$ normalized PPs, all obtained using the trace polynomial. Note that all these $s$-polynomials could very well generate more PPs whilst using some other rank 1 linearized polynomial for $s$ instead of the trace function. Such PPs have not been accounted for in this example.  %Out of the remaining $2800$, all have a count of $252$ PPs of the form given in \tref{thm:Fp_to_Fosquare_comp_inv} associated with them, except for the $102$ which satisfy \tref{thm:from_perm}. These $102$ polynomials have exactly one nonzero $\gamma x$ which makes $g(x)+\gamma x$ a PP of $\mathbb{F}_7$. Therefore there are $294$ additional PPs associated with them, making the count $546$ for them. 
\end{example}
%%%%%%%%%%%%%%%%%%%%%%%%%%%%%%%%%%%%%%%
Continuing with the trace function, if we look at more general $s$-polynomials whose coefficients may come from $\mathbb{F}_{q^2}\setminus\mathbb{F}_q$, then the situation becomes more complicated as the image of $g(s)$ is no longer restricted to a subspace of $\mathbb{F}_{q^2}$. Before exploring these polynomials, we first look at $s$-polynomials $g(s) = s^m + \Sigma_{i=2}^{m-1} a_i s^i$ with coefficients $a_i \in \mathbb{F}_q$ that can form normalized PPs with rank 1 linear maps $L(x)$, whilst $s = x^q + x$ is the trace polynomial. As it turns out the condition is very similar to \tref{thm:coprime}.

\begin{theorem}\label{thm:rnk1coprime}
If $g(x) = x^m + \Sigma_{i=2}^{m-1} a_ix^i$ with $a_i \in \mathbb{F}_q$ is a PP of $\mathbb{F}_q$, then there are $q^2(q-1)$ normalized PPs of $\mathbb{F}_{q^2}$ having the form $f = g(x^q+x) +L(x)$, where $L(x)$ is a rank 1 linear map.
\end{theorem}

\begin{proof}
The proof runs along the lines of the proof of \tref{thm:coprime}. As $g(x)$ is a permutation polynomial of $\mathbb{F}_q$ and the image of the trace function $x^q+x$ is $\mathbb{F}_q$, therefore the image of $g(x^q+x)$ is precisely equal to $\mathbb{F}_q$. As a result, if one takes a rank 1 linear map $L(x)$ whose image is a one dimensional subspace which is not $\mathbb{F}_q$ and whose kernel is not the same as the kernel of the trace function, then the image of $f = g(x^q+x) + L(x)$ has cardinality $q^2$ and the kernel of $f$ is only $\{0\}$.

The number of monic linearized polynomials $L(x)$ of rank $1$ whose kernel is not the same as that of the trace function is $q$. Each of these monic linearized polynomials can be multiplied by $\eta \in \mathbb{F}_{q^2}^*$ to give us other rank 1 $L(x)$. As we need to ensure that the image of $L(x)$ is not $\mathbb{F}_q$, therefore the number of choices of $\eta$ is restricted to $q^2-q$ for each monic rank 1 linearized polynomial. Thus, the total number of normalized PPs of the form $g(x^q+x) + L(x)$ with rank 1 $L(x)$ is $q^2(q-1)$.  
\end{proof}

\begin{example}
Continuing with the earlier example, there are $15$ monic PPs of $\mathbb{F}_7$ having the form $g(x) = x^m + \Sigma_{i=2}^{m-1} a_ix^i \in \mathbb{F}_7[x]$. All of them generate $294$ PPs of $\mathbb{F}_{7^2}$, of the form   $g(x^7+x)+L(x)$ with rank $1$ $L(x)$.
\end{example}

Using \cref{cor:pp_from_base}, we can state the equivalent result for other $s = x^q + \delta_ix$, where $\delta_i$ is a $(q+1)$-th root of unity.

\begin{corollary}\label{cor:gen_srnk1}
If $x^m + \Sigma_{i=2}^{m-1} a_ix^i$ with $a_i \in \mathbb{F}_q$ is a PP of $\mathbb{F}_q$, then every $s$-polynomial of the form $g(s) = s^m + \Sigma_{i=2}^{m-1} \lambda^{m-i}a_is^i$, where $s = x^q + \delta_ix$ is a rank 1 linearized polynomial and $\lambda$ is in the image of $s$, produces $q^2(q-1)$ normalized PPs of the form $f = g(s) + L(x)$, where $L(x)$ is a rank 1 linearized polynomial. 
\end{corollary}
\begin{proof}
We re-write $g(s)$ as $\lambda^mg'(\lambda^{-1}s)$ where $g'(x) = \Sigma_{i=2}^m a_ix^i$. As $\lambda$ is in the image of $s = x^q + \delta_ix$, therefore the image of $\lambda^{-1}s$ lies in $\mathbb{F}_q$. Since $g'(x)$ is a PP of $\mathbb{F}_q$, therefore the image of $g'(\lambda^{-1}s)$ is precisely equal to $\mathbb{F}_q$ and consequently the image of $\lambda^mg'(\lambda^{-1}s)$ is precisely equal to the span of $\lambda^m$. As a result, if one takes any rank $1$ linear map $L(x)$ whose image is not the span of $\lambda^m$ and whose kernel is not the same as the kernel of $s$, then image of  $g(s) + L(x)$ has cardinality $q^2$ and its kernel contains only $\{0\}$ and therefore is a PP of $\mathbb{F}_{q^2}$. The count follows by using the same arguments as in \tref{thm:rnk1coprime}.
\end{proof}
\begin{example}
Over $\mathbb{F}_{11^2}$, the polynomial $g(x)=x^6+x^5+5x^4+9x^3+2x^2$ gives $1210$ PPs of the form $g(x^{11}+x) + L(x)$ where $L(x)$ is a rank $1$ linearized polynomial. $g(x)$ is a PP of $\mathbb{F}_{11}$. For the rank $1$ linearized polynomial $s_i=x^{11}+a^{10}x$, one can modify $g(x)$, to $g_1(x)=x^6+ a^{23} x^5+5a^{46} x^4+9a^{69} x^3+2a^{92} x^2$ as $a^{23} \in {\rm im}\,(s_i)$. Then one obtains $1210$ PPs of the form $g_1(x^{11}+a^{10}x) + L(x)$ where $L(x)$ is a rank $1$ linearized polynomial.
\end{example}

We now provide the algorithm to find the compositional inverses of PPs of the form $g(s) + L(x)$, where $L(x)$ is a rank $1$ linearized polynomial.

\begin{corollary}
Let $g(x) = x^m + \Sigma_{i=2}^{m-1}a_ix^i$ with $a_i \in \mathbb{F}_q$ be a PP of $\mathbb{F}_q$. Let $s = x^q + \delta_ix$ be a rank $1$ linearized polynomial and $\lambda \in {\rm im}\, s$. Then the compositional inverse of a PP of $\mathbb{F}_{q^2}$ of the form $f(x) = \lambda^mg(\lambda^{-1}s) + L(x)$ where $L(x)$ is a rank $1$ linearized polynomial, is of the form $h(x) = \kappa\psi(\lambda_1^{-1}s_j) + M(x)$ where $s_j$ and $M(x)$ are rank $1$ linearized polynomials, $\psi(x)$ is the compositional inverse of $g(x)$ over $\mathbb{F}_q$, $\lambda_1 \in {\rm im}\, s_j$ and $\kappa \in \mathbb{F}_{q^2}$.
\end{corollary}
\begin{proof}
Consider a basis $\{u,v\}$ of $\mathbb{F}_{q^2}$ where $u \in \ker s$ and $v \in \ker L(x)$. If $w$ is a generic element of $\mathbb{F}_{q^2}$, then  $w = \alpha u + \beta v$ with $\alpha, \beta \in \mathbb{F}_q$. Then $f(w) = \lambda^mg(\lambda^{-1}s(\beta v)) + L(\alpha u) = \lambda^mg(\beta z) + \alpha L(u)$ where $z = \lambda^{-1}s(v) \in \mathbb{F}_q$. Note that as $\lambda \in {\rm im}\,s$, therefore $\lambda^{-1}s(\beta v) \in \mathbb{F}_q$. 

Let $s_j = x^q + \delta_jx$ be a rank $1$ linearized polynomial such that $L(u) \in \ker s_j$. Let $\eta = 1/s_j(\lambda^m)$. Then $\eta s_j(f(w)) = \eta s_j(\lambda^m g(\beta z)) = g(\beta z)$. Choose rank $1$ linearized polynomial $M(x)$ such that $\lambda^m \in \ker M$ and $M(L(u)) = u$.  
As $g(x)$ is a PP of $\mathbb{F}_q$, it has a compositional inverse $\psi(x)$ which is a PP of $\mathbb{F}_q$. Then $h(x) = z^{-1}v \psi(\eta s_j) + M(x)$ is the compositional inverse of $f(x)$. This follows from \begin{eqnarray*} h(f(w)) = h(\lambda^m g(\beta z) + \alpha L(u)) & = & (z^{-1}v \psi(\eta s_j) + M)(\lambda^mg(\beta z) + \alpha L(u)) \\ & = & z^{-1}v\psi(g(\beta z)) + \alpha M(L(u)) \\ & = & z^{-1}v\beta z + \alpha u \\ & = & \alpha u + \beta v = w \end{eqnarray*}
Note that $\eta = 1/s_j(\lambda^m)$ is indeed of the form $\lambda_1^{-1}$ where $\lambda_1 \in {\rm im}\,s_j$ and $z^{-1}v$ plays the role of $\kappa$ in the statement of the corollary.
\end{proof}
We have hereby completed enumeration of all normalized PPs that arise from $s$-polynomials with coefficients from $\mathbb{F}_q$ for the trace polynomial combining with linearized polynomials of ranks $1$ and $2$. We have also indicated the equivalent $s$-polynomials for other instantiations of rank $1$ linearized polynomials $s = x^q + \delta_ix$ that result in similar enumerations. Explicit constructions of the compositional inverses of these normalized PPs were also obtained.  \\

%%%%%%%%%%%%%%%%%%%%%%%%%%%%%%%%%%%%%%%%%%%%%%%%%%%%%%%%%%%%%%%%%%%

Before we tackle the cases of most general $s$-polynomials, we now state a significant property that all normalized PPs of the form $f = g(s) + L(x)$ share. 

\begin{lemma}\label{lem:affine}
All normalized PPs of the form $f = g(s) +L(x)$ where $s = x^q + \delta_ix$ is a rank $1$ linearized polynomial and $L(x)$ is a linearized polynomial, maps a particular subspace $\mathcal{U}$ and affine spaces associated with it of the form $v + \mathcal{U}$ to a subspace $\mathcal{V}$ and affine spaces associated with it of the form $b + \mathcal{V}$. 
\end{lemma}

\begin{proof}
The kernel of $s = x^q + \delta_ix$ forms the subspace $\mathcal{U}$ and the affine spaces associated with it can be thought of as $\alpha v + \mathcal{U}$, where $v \not \in \mathcal{U}$ and $\alpha \in \mathbb{F}_q$. Under the action of a normalized PP of the form $f = g(s) + L(x)$, $\mathcal{U}$ would map to $L(\mathcal{U}) = \mathcal{V}$. Any element $w$ in the affine space $\alpha v + \mathcal{U}$ would be of the form $w = \alpha v + u$ where $u \in \mathcal{U}$. Then $f(w) = g(s(w)) + L(w) = g(\alpha s(v)) + \alpha L(v) + L(u) \in b +\mathcal{V}$, where $b = g(\alpha s(v)) + \alpha L(v)$.  
\end{proof}

Using \lref{lem:affine} we can now characterize all $s$-polynomials in $\mathbb{F}_{q^2}[x]$ that give rise to normalized PPs of the form $f = g(s) + L(x)$. We shall first deal with the case where $L(x)$ is a rank 1 linearized polynomial.

\begin{theorem}\label{thm:allrnk1}
Consider any polynomial $g(x) = x^m + \Sigma_{i=2}^{m-1} a_i x^i \in \mathbb{F}_{q^2}[x]$. Let $s = x^q + \delta_ix$ and $v \in \mathbb{F}_{q^2}$ such that $0 \not = v \in {\rm im }\,\, s$.  Let $\mathfrak{G}$ be the collection of all elements that are differences of images under $g$ of distinct elements in span of $v$, that is $\mathfrak{G} = \{ g(\alpha v) - g(\beta v) \, \mid \, \alpha ,\beta \in \mathbb{F}_q, \alpha \not = \beta \}$. For every one dimensional subspace $\mathcal{L} \subset \mathbb{F}_{q^2}$ such that $\mathfrak{G} \cap \mathcal{L} = \emptyset$, there exists $q(q-1)$ rank 1 linearized polynomials $L(x)$ such that $f = g(s) + L(x)$ is a normalized PP of $\mathbb{F}_{q^2}$. Conversely, if $\mathfrak{G} \cap \mathcal{L} \not = \emptyset$ for every one dimensional subspace $\mathcal{L}$ of $\mathbb{F}_{q^2}$, then there exists no rank $1$ linearized polynomial $L(x)$ such that $f = g(s) + L(x)$ is a normalized PP of $\mathbb{F}_{q^2}$.
\end{theorem}

\begin{proof}
Assume $f = g(s) + L(x)$ is a normalized PP of $\mathbb{F}_{q^2}$ where $L(x)$ is a rank $1$ linearized polynomial. Clearly, the kernels of $s$ and $L(x)$ cannot be the same subspace of $\mathbb{F}_{q^2}$. Let $\{u,v_1\}$ form a basis of $\mathbb{F}_{q^2}$ where $u \in \ker s$ and $v_1 \in \ker L(x)$. Then $L(u) \not = 0$. Let  $\alpha_1,\alpha_2 \in \mathbb{F}_q$ such that $g(s(\alpha_1 v_1)) = a_1$ and $g(s(\alpha_2 v_1)) = a_2$. If $a_1 - a_2 \in {\rm im }\, L(x)$, then one can find $\beta \in \mathbb{F}_q$ such that $L(\beta u) = a_1 - a_2$. Then $f(\alpha_1 v_1) = a_1 = a_2 + L(\beta u) = f(\alpha_2 v_1 + \beta u)$ and thus two distinct elements $\alpha_1 v_1$ and $\alpha_2 v_1 + \beta u$ both map to the same element $a_1$ under $f$ thereby contradicting the fact that $f = g(s) + L(x)$ is a PP. Therefore $\mathfrak{G} \cap \mathcal{L}$ must be the empty set where $\mathcal{L}$ is the image of $L(x)$. Clearly if $\mathfrak{G} \cap \mathcal{L} \not = \emptyset$ for every subspace $\mathcal{L}$, then we cannot find a rank $1$ linearized polynomial that would combine with $g(s)$ to form a PP.

Conversely, let $\mathcal{L}$ be a one dimensional subspace such that $\mathfrak{G} \cap \mathcal{L} = \emptyset$. Observe that $g(s(0)) = 0$ and so $0$ is in the image of $g(s)$. This implies that all nonzero elements in the image of $g(s)$ are guaranteed to be in $\mathfrak{G}$. Now $0 \in \mathcal{L}$ for every subspace $\mathcal{L}$. Therefore $0 \not \in \mathfrak{G}$. In other words, $g(\alpha_1v) \not = g(\alpha_2v)$ for distinct $\alpha_1,\alpha_2 \in \mathbb{F}_q$ and $v \in {\rm im}\, s$. This guarantees that the cardinality of the set  im $g(s)$ is precisely $q$. If we now choose any rank 1 linearized polynomial $L(x)$ whose image is $\mathcal{L}$ and $L(u) \not = 0$, then $f = g(s) + L(x)$ is a PP. This follows from the argument used in the proof of \tref{thm:coprime} : the cardinality of the image of $L(x)$ is also $q$ and the sum of the images of $g(s)$ and $L(x)$ gives $q^2$ distinct elements. 

There are $q$ monic linearized polynomials $x^q + \delta_jx$ which are rank 1 and satisfying the inequality $u^q + \delta_ju \not = 0$, where $s(u) = 0$. To align the image of these monic linearized polynomials to the subspace $\mathcal{L}$, one has a choice of $(q-1)$ elements $\eta \in \mathbb{F}_{q^2}^*$ for each of the monic rank 1 linearized polynomials, that is $L(x) = \eta(x^q + \delta_jx)$. Thus, one obtains $q(q-1)$ rank 1 linearized polynomials $L(x)$ for every subspace $\mathcal{L}$ that satisfies the condition $\mathfrak{G} \cap \mathcal{L} = \emptyset$. 
\end{proof}

The above theorem is a generalization of all the results in \tref{thm:coprime}, \cref{cor:coprime_count}, \tref{thm:rnk1coprime} and \cref{cor:gen_srnk1}. The number of rank 1 linearized polynomials that appeared in \cref{cor:coprime_count} was $q^2(q-1)$. This was due to the fact that the image of the monomial in \tref{thm:coprime} was precisely equal to a subspace of $\mathbb{F}_{q^2}$ and therefore all elements in $\mathfrak{G}$ belongs to that subspace of $\mathbb{F}_{q^2}$. Hence all the other $q$ one dimensional subspaces are eligible candidates for $\mathcal{L}$ in \tref{thm:allrnk1}. Similarly, the image of $g(s)$ in \tref{thm:rnk1coprime} and \cref{cor:gen_srnk1} is precisely equal to a subspace of $\mathbb{F}_{q^2}$. As a result, in those cases again, all other one dimensional spaces are eligible candidates for $\mathcal{L}$ of \tref{thm:allrnk1}. 

We now characterize all $s$-polynomials in $\mathbb{F}_{q^2}[x]$ that give rise to normalized PPs of the form $f = g(s) + L(x)$ where $L(x)$ is a rank $2$ linearized polynomial.

\begin{theorem}\label{thm:allrnk2}
Consider any polynomial $g(x) = x^m + \Sigma_{i=2}^{m-1} a_i x^i \in \mathbb{F}_{q^2}[x]$. Let $s = x^q + \delta_ix$ and $v \in \mathbb{F}_{q^2}$ such that $0 \not = v \in {\rm im }\, s$.  Let $\mathfrak{H}$ be the collection of all elements obtained in the following fashion $\mathfrak{H} = \{ \frac{g(\alpha v) - g(\beta v)}{\beta - \alpha} \, \mid \, {\rm where  }\,\, \alpha, \beta \in \mathbb{F}_q, \alpha \not = \beta \}$. For every one dimensional subspace $\mathcal{L} \subset \mathbb{F}_{q^2}$, consider  affine spaces $b + \mathcal{L}$. For every nontrivial affine space $b + \mathcal{L}$ such that $\mathfrak{H} \cap (b + \mathcal{L}) = \emptyset$, there exists $q(q-1)$ rank $2$ linearized polynomials $L(x)$ such that $f = g(s) + L(x)$ is a normalized PP of $\mathbb{F}_{q^2}$. Conversely, if $\mathfrak{H} \cap (b + \mathcal{L}) \not = \emptyset$ for every one dimensional affine space $b + \mathcal{L}$ of $\mathbb{F}_{q^2}$, then there exists no rank $2$ linearized polynomial $L(x)$ such that $f = g(s) + L(x)$ is a normalized PP of $\mathbb{F}_{q^2}$.
\end{theorem}

\begin{proof}
Let $f = g(s) + L(x)$ be a PP of $\mathbb{F}_{q^2}$. Let $\{u,v_1\}$ be a basis of $\mathbb{F}_{q^2}$ such that $u \in \ker s$ and $s(v_1) = v$, where $v \not = 0$ was used to generate $\mathfrak{H}$ as given in the statement of the theorem. Let $\mathcal{L}$ be the one dimensional subspace of $\mathbb{F}_{q^2}$ spanned by $L(u)$. Consider any $w \in \mathbb{F}_{q^2}$ of the form $w = v_1 + \alpha u$ where $\alpha \in \mathbb{F}_q$. Consider two elements $\beta_1 w, \beta_2 w$ in the subspace spanned by $w$. \[ f(\beta_i w) = g(s(\beta_i w)) + L(\beta_i w) = g(\beta_i s(w)) + \beta_i L(w) = g(\beta_i v) + \beta_i L(w) \]
for $i = 1,2$. As $f$ is a PP, $f(\beta_1w) \not = f(\beta_2w)$. Using the expansion above, this inequality implies that $L(w) \not = \frac{g(\beta_1 v) - g(\beta_2 v)}{\beta_2 - \beta_1} \in \mathfrak{H}$. Thus we can conclude that the affine space defined by all $L(w) = L(v_1 + \alpha u)$ which is given as $L(v_1) + \mathcal{L}$ does not intersect the set $\mathfrak{H}$.

Conversely, assume that some nontrivial affine space $b + \mathcal{L}$ satisfies the property that $\mathfrak{H} \cap (b + \mathcal{L}) = \emptyset$. By nontrivial affine set, we mean that the set is not a subspace, that is, $(b + \mathcal{L}) \not = \mathcal{L}$. A linear map $L(x)$ is completely  defined by defining where a basis of $\mathbb{F}_{q^2}$ maps to. Taking the basis $\{u,v_1\}$ with $u \in \ker s$ and $s(v_1) = v$ as before, we define $0 \not = L(u) \in \mathcal{L}$ and $L(v_1) = b$. Then we claim that the polynomial $f = g(s) + L(x)$ is a permutation polynomial. 

If $w_i = v_1 + \alpha_i u$, then we need to show that $f(\beta_1 w_1) \not = f(\beta_2 w_2)$. If $\beta_1 = \beta_2$ and $w_1 \not = w_2$, then $f(\beta_1w_1) - f(\beta_1w_2) = L(\beta_1 w_1) - L(\beta_1 w_2) = \beta_1(\alpha_1 - \alpha_2)L(u) \not = 0$ as $\alpha_1 \not = \alpha_2$ and $L(u) \not = 0$. For the case $\beta_1 \not = \beta_2$, let us suppose $f(\beta_1 w_1) = f(\beta_2 w_2)$. Then rearranging the terms, we get \begin{eqnarray*} g(\beta_1 v) - g(\beta_2 v) & = & \beta_2 L(w_2) - \beta_1 L(w_1)  \\ & = & L((\beta_2 - \beta_1)v_1 + (\beta_2\alpha_2 - \beta_1\alpha_1)u)  \\ & = & (\beta_2 - \beta_1)L(v_1 + \frac{\beta_2\alpha_2 - \beta_1\alpha_1}{\beta_2 - \beta_1}u) \\ \Rightarrow \frac{g(\beta_1v) - g(\beta_2v)}{\beta_2 - \beta_1} & = & L(w_j) \,\,\,\,\, \mbox{\rm where }\,\, L(w_j) \in (b + \mathcal{L}) \end{eqnarray*} The above equality contradicts the fact that $\mathfrak{H} \cap (b + \mathcal{L}) = \emptyset$ and therefore $f$ is a PP. Observe that there are $q-1$ choices for $L(u) \in \mathcal{L}$. Also observe that there are $q$ distinct choices for $v_1$, such that $s(v_1) = v$. Thus we can obtain $q(q-1)$ distinct rank $2$ linearized polynomials $L(x)$ all of which form PP of the form $f = g(s) + L(x)$.
\end{proof}

Interestingly, while the most general conditions that the $s$-polynomial should satisfy to combine with rank $1$ linearized polynomials to form a PP  (\tref{thm:allrnk1}) depends on a subspace condition, the most general conditions that the $s$-polynomials should satisfy to combine with rank $2$ linearized polynomials to form a PP (\tref{thm:allrnk2}) hinges on a affine space condition.  Observe that \tref{thm:allrnk2} is a generalization of all the results in \tref{thm:lowerbound}, \cref{cor:lowerbound}, \tref{thm:Fp_to_Fpsquare}, \tref{thm:pp_from_basefield_poly}, \cref{cor:pp_from_base} and \tref{thm:from_perm}. For example, in \tref{thm:pp_from_basefield_poly}, as the image of the $s$-polynomial lies within $\mathbb{F}_q$, therefore in this case, the set $\mathfrak{H} \subset \mathbb{F}_q$. This ensures that for all nontrivial affine sets associated to $\mathbb{F}_q$, that is, $b + \mathbb{F}_q$, we have $\mathfrak{H} \cap (b + \mathbb{F}_q) = \emptyset$. Since there are $(q-1)$ distinct affine sets of this kind, there are a total of $q(q-1)^2$ linearized polynomials $L(x)$ that can combine with a given $s$-polynomial in \tref{thm:pp_from_basefield_poly} to form PPs of $\mathbb{F}_{q^2}$. 
%%%%%%%%%%%%%%%%%%%%%%%%%%%%%%%%%%%
\section{Conclusion}
This paper addresses the construction and analysis of permutation polynomials of the form $g(s) + L(x)$ over $\mathbb{F}_{q^2}$ where $s, L(x)$ are linearized polynomials and $g(x) \in \mathbb{F}_{q^2}[x]$. We consider $s$ to be a rank $1$ linearized polynomial over $\mathbb{F}_{q^2}$ and therefore of the form $(x^q+\delta_i x)$, where $\delta_i$ is a $(q+1)$-th root of unity. In this paper we gave conditions for polynomials of the form $f(x)=g(s)+L(x)$ to be normalized PPs of $\mathbb{F}_{q^2}$; where $L(x)$ is also a linearized polynomial of rank $1$ or $2$. 
We have divided these normalized PPs into subclasses based on some conditions and enumerated the number of PPs that arise in each of these subclasses. We have also demonstrated how to construct compositional inverses for most of these subclasses. All the arguments used in this paper are based on linear algebra. The class of PPs that eventually fall into the ambit of this paper include as subsets a large number of classes of PPs that have been studied in literature. It is our belief that such a linear algebraic approach is amenable to generalizations applicable to more general fields $\mathbb{F}_{q^n}$.
%\end{linenumbers}
%%%%%%%%%%%%%%%%%%%%%%%%%%%%%%%%%%%%%%%%%%%%%%%5
\bibliographystyle{elsarticle-harv} 
\bibliography{MK_HP_2}
\end{document}